\documentclass[11pt]{article}
\usepackage{amsmath,amssymb,amsthm}
\usepackage{hyperref} 
\usepackage{tikz-cd} 
\theoremstyle{definition} 
\newtheorem{definition}{Definition}[section]
\newtheorem{remark}[definition]{Remark}
\newtheorem{example}[definition]{Example} 
\theoremstyle{plain}
\newtheorem{theorem}[definition]{Theorem}
\newtheorem{proposition}[definition]{Proposition}

\newtheorem{lemma}[definition]{Lemma}

\title{Nef but Non-Semi-positive Line Bundles on Hopf Manifolds}
\author{Xiaojun Wu} 
\date{\today} 
\begin{document} 
\def\A{\mathcal{A}}
\def\cI{\mathcal{I}}
\def\Z{\mathbb{Z}}
\def\Q{\mathbb{Q}}  \def\C{\mathbb{C}}
 \def\R{\mathbb{R}}
 \def\N{\mathbb{N}}
 \def\H{\mathbb{H}}
  \def\P{\mathbb{P}}
 \def\rC{\mathrm{C}}
  \def\d{\partial}
 \def\dbar{{\overline{\partial}}}
\def\dzbar{{\overline{dz}}}
\def \ddbar {\partial \overline{\partial}}
\def\cB{\mathcal{B}}
\def\cD{\mathcal{D}}  \def\cO{\mathcal{O}}
\def\cbarO{\overline{\mathcal{O}}}
\def\D{\mathcal{D}}
\def\cC{\mathcal{C}}
\def\cF{\mathcal{F}}
\def \rank{\mathrm{rank}}
\def \deg{\mathrm{deg}}
\def \tot{\mathrm{Tot}}
\def \id{\mathrm{id}}
\bibliographystyle{plain}
\def \End{\mathrm{End}}
\def \dim{\mathrm{dim}}
\def \div{\mathrm{div}}
\def \ker{\mathrm{Ker}}
\def \im{\mathrm{Im}}
\def \rC{\mathrm{Cone}}
\newcommand{\Ub}{\mathcal{U}}
\newcommand{\dcech}{\check{\delta}}
\newcommand{\dcechg}{\delta\!\!\!\check{\delta}}
\newcommand{\lc}{\mathcal{L}}
\newcommand{\ec}{\mathcal{E}}
\maketitle
\begin{abstract}
We give a complete characterization of the nef, effective, pseudo-effective, and semi-positive cones of Hopf manifolds. In particular, by computing the Ueda classes on non-diagonal Hopf surfaces, we show that the invariant elliptic curve is nef but not semi-positive. Via Bott--Chern cohomology calculations, we construct new examples of nef but non-semi-positive line bundles on Hopf manifolds of arbitrary dimension.
We also prove that a minimal compact complex surface \(X\) containing an elliptic curve \(C\) and a nef line bundle \(L\) such that \((X,C,L)\) has finite generalized Ueda type is either a Hopf surface or Serre's example.
\end{abstract}
\section{Introduction}

It is a fundamental problem in complex geometry to understand the relationship
between numerical effectiveness and curvature positivity of holomorphic line
bundles.
Motivated by Fujita's question \cite{Fuj83}, one may ask whether every nef
line bundle admits a smooth Hermitian metric with semi-positive Chern
curvature.
This expectation was disproved by Demailly, Peternell and Schneider
\cite[Example~1.7]{DPS94}, who constructed nef line bundles that are not
semi-positive (see also \cite{KO70, Yau74}).

In \cite{Ued83}, T. Ueda studied neighborhoods of compact complex curves with a topologically trivial normal bundle, introducing cohomological invariants, now called Ueda classes, which measure the obstruction to analytically linearizing such neighborhoods. These classes arise from the formal linearization of transition functions along the submanifold and detect when a formally flat structure cannot be integrated analytically. While Ueda's original work focused on curves, subsequent developments by T. Koike \cite{Koi20,Koi21} extended Ueda theory to higher-codimensional smooth submanifolds with unitary flat normal bundles. In this generalized setting, the obstruction classes have been applied to derive criteria relevant to semi-positivity problems for Hermitian line bundles restricted to these submanifolds \cite[Theorem 1.4]{Koi21} and \cite[Theorem 1.2]{Koi22} (see also \cite[Theorem 1.1]{Koi15}, \cite[Corollary 2.5]{Wu26}).

In \cite[Article 2, Section 9]{Nee}, Neeman investigated the nontriviality of Ueda classes associated with an elliptic curve in a complex surface. He showed that, among minimal non-algebraic surfaces, Hopf surfaces are the only ones for which such nontrivial Ueda classes can occur
(see Proposition \ref{prop-class}).

For the invariant elliptic curve $D$ on a non-diagonal Hopf surface $X$, Neeman \cite{Nee} stated the corresponding Ueda classes without proof. In Section~\ref{Hopf-surf}, we give an explicit computation of these Ueda classes. These coincide with the Ueda classes of the line bundle $\mathcal O_X(D)$ in the sense of Koike \cite{Koi21}.
By \cite[Theorem~1.2]{Koi22}, the non-vanishing of the Ueda class at some degree implies that \(\mathcal O_X(D)\) is nef but not semi-positive. In fact, \(\mathcal O_X(D)\) admits no smooth Hermitian metric with semi-positive Chern curvature on any neighborhood of \(D\).

A natural question is whether $\mathcal O_X(D)$ is the only nef but
non-semi-positive line bundle on a non-diagonal Hopf surface.
More generally, one may ask for a characterization of nef, effective, 
pseudo-effective, and semi-positive line bundles on Hopf manifolds.

Since every holomorphic line bundle on a Hopf manifold is flat by a theorem
of Mall \cite{Mall91}, this problem reduces to understanding the
Bott--Chern first Chern classes of flat line bundles.
Our first main result gives an explicit relation among these Chern classes.
More precisely, if $L_\delta$ denotes the flat line bundle corresponding to
$\delta\in\C^*$, then
\[
\log|\delta_2|\,c_1(L_{\delta_1})
=
\log|\delta_1|\,c_1(L_{\delta_2})
\]
(Proposition \ref{prop:chern-relation}).
Consequently, the Bott--Chern first Chern class of a flat line bundle is
completely determined by the modulus $|\delta|$.
As a consequence, combining our formula with the computation of the Bott--Chern cohomology of Hopf manifolds in \cite{IO25}, we obtain explicit descriptions of the nef, effective, and pseudo-effective cones and characterize the nef, effective and pseudo-effective line bundles (see Theorem \ref{thm:hopf-surface-cones}).

To determine the semi-positive cone, one must distinguish between diagonal and
non-diagonal Hopf surfaces. For diagonal Hopf surfaces, using the explicit
construction of Hermitian metrics in the proof of
\cite[Proposition~6.4]{DPS94}, every nef line bundle is
semi-positive. In contrast, Koike's non-semi-positivity criterion implies
that every nontrivial nef line bundle on a non-diagonal Hopf surface fails to
admit a smooth Hermitian metric with semi-positive Chern curvature.
Consequently, the semi-positive cone coincides with the nef cone in the
diagonal case, whereas it is trivial in the non-diagonal case.

The higher-dimensional case is considerably more subtle. Exploiting the
natural flag of Hopf submanifolds together with the Bott--Chern cohomology of
Hopf manifolds, we obtain a complete description of the nef, effective, 
pseudo-effective, and semi-positive cones of arbitrary primary Hopf manifolds (Proposition \ref{cone-hopf-mfd}).
One of the main ingredients is the recent computation of the Bott--Chern cohomology of primary Hopf manifolds in \cite{IO25}. We also give an alternative proof of the description of the nef, effective,  pseudo-effective, and semi-positive line bundles (Remark \ref{rem: Paun}), based on a result of P\u{a}un \cite{Paun98}. This approach also yields the corresponding classification for secondary Hopf manifolds (Remark \ref{rem: secondary-Hopf}).

In summary, the main results of this paper can be summarized as follows.

\medskip

\noindent
\textbf{Main Theorem} (=Theorem \ref{thm:hopf-surface-cones} and Proposition \ref{cone-hopf-mfd})
Let \(X\) be a primary Hopf manifold of dimension \(n\) in the normal form of Lemma \ref{normal-form}, and consider the natural flag
\[
X=Z_0\supset Z_1\supset\cdots\supset Z_{n-1},
\]
where
\(
Z_i, i\leq n-1,
\)
are primary Hopf manifolds of dimension \(n-i\), and \(Z_{n-1}\) is an elliptic curve. Then the nef, effective, and pseudo-effective cones of \(X\) consist of first Chern classes in Bott--Chern cohomology of flat line bundles \(L_\delta\) corresponding to the representations
\[
\rho_\delta:\pi_1(X)\simeq\mathbb{Z}\longrightarrow\mathbb{C}^*,
\qquad
\rho_\delta(1)=\delta,
\]
where the generator of \(\mathbb{Z}\) is chosen so that \(\mathcal{O}_X(Z_1)\) corresponds to some \(\delta_0\) with \(|\delta_0|<1\).

On the other hand, if \(Z_{n-2}\) is a non-diagonal Hopf surface, then the semi-positive cone of \(X\) is trivial, whereas if \(Z_{n-2}\) is a diagonal Hopf surface, then the semi-positive cone of \(X\) coincides with the nef cone.
\medskip

\noindent

We also put the Hopf surface case into the broader context of Ueda theory on compact complex surfaces. In the appendix, we prove that if $X$ is a minimal compact complex surface containing an elliptic curve $C$ and a nef line bundle $L$ such that the Ueda type of $(X,C,L)$ is finite (see Definition \ref{ueda-def}), then $X$ is either a Hopf surface or Serre's example. This extends the corresponding classification of Neeman from the case $L=\mathcal{O}_X(C)$ to arbitrary nef line bundles. We also show that a minimal non-projective surface cannot contain a curve of genus at least $2$ with self-intersection zero, and give a higher-dimensional construction of non-minimal projective manifolds carrying smooth divisors of finite Ueda type whose associated line bundles are nef but not semi-positive (see Example \ref{example-decomp}).

 \paragraph{}
\textbf{Acknowledgements.} I would like to thank my postdoctoral advisor, Professor Takayuki Koike for several helpful and stimulating discussions related to this work. This research was supported by the JSPS Postdoctoral Fellowships for Research in Japan (Standard). I am also grateful to Osaka Metropolitan University and the Osaka Central Advanced Mathematical Institute (OCAMI) for providing an excellent research environment.

\section{Positive cones on Hopf surfaces}
\label{Hopf-surf}
We recall the following classical result of Kodaira (\cite[Theorem 30]{Koi66}) and Poincaré--Dulac concerning the normal form of contractions defining primary Hopf surfaces.

\begin{proposition}[Normal form of primary Hopf surfaces]
Let
\[
X=(\mathbb{C}^2\setminus\{0\})/\langle\gamma\rangle
\]
be a primary Hopf surface, where $\gamma \in \mathrm{Aut}(\mathbb{C}^2 \setminus \{0\})$ is a contraction mapping fixing the origin. 

Then, after a holomorphic change of coordinates, the generator
\(\gamma\) is conjugated to
\[
\gamma(z_1,z_2)=(\alpha z_1+\lambda z_2^m,\beta z_2),
\]
where
\[
0<|\alpha|\leq |\beta|<1,\qquad m\in\mathbb{N},
\]
and
\[
\lambda(\beta^m-\alpha)=0.
\]
Moreover, \(X\) is called diagonal if and only if \(\lambda=0\). Otherwise, \(X\) is called non-diagonal, and in this case, the resonance condition yields
\(
\alpha=\beta^m
\)
for some integer \(m\geq1\).
\end{proposition}
Recall the following consequence of the explicit construction of Hermitian metrics in the proof of
\cite[Proposition~6.4]{DPS94}.

\begin{lemma}

Let \(X\) be a diagonal primary Hopf surface. Then the two coordinate axes
\[
\{z_1=0\},\qquad \{z_2=0\}
\]
descend to elliptic curves \(C,D\subset X\). Moreover,
\(
-K_X=C+D,
\)
and the line bundles \(\mathcal O_X(C)\) and \(\mathcal O_X(D)\) admit smooth Hermitian metrics with semi-positive Chern curvature. In particular, the anticanonical line bundle \(\mathcal O_X(-K_X)\) is semi-positive.
\end{lemma}

\begin{proof}
The construction in the proof of \cite[Proposition~6.4]{DPS94} gives \(C^2\) Hermitian metrics on the line bundles \(\mathcal O_X(C)\) and \(\mathcal O_X(D)\). We modify this construction to obtain smooth Hermitian metrics with semi-positive Chern curvature.

Set
\[
c:=\frac{|\beta|}{|\alpha|}\geq 1.
\]
A slight modification of the construction in \cite[Proposition~6.4]{DPS94} gives \(C^2\) Hermitian metrics whose local expressions on \(\mathbb C^2\setminus\{0\}\) are
\[
h_C:=\bigl(|z_1|^{2c}+|z_2|^2\bigr)^{-1}|dz_2|^2,
\qquad
h_D:=\bigl(|z_1|^2+|z_2|^{2/c}\bigr)^{-1}|dz_1|^2
\]
for \(\mathcal O_X(C)\) and \(\mathcal O_X(D)\), respectively. The corresponding local weights are smooth on \(\{z_1\ne0\}\) and \(\{z_2\ne0\}\), respectively.

Fix a smooth increasing function
\(
\rho\colon[0,1]\longrightarrow[0,1]
\)
which is equal to \(0\) near \(0\) and to \(1\) near \(1\). Consider the smooth \((1,1)\)-form
\[
\Theta
=
\rho\left(\frac{|z_1|^{2c}}{|z_1|^{2c}+|z_2|^2}\right)
\sqrt{-1}\Theta(h_C)
+
\left(
1-\rho\left(\frac{|z_1|^{2c}}{|z_1|^{2c}+|z_2|^2}\right)
\right)
\sqrt{-1}\Theta(h_D).
\]
Since the cutoff function is constant near the two coordinate axes, the form \(\Theta\) is smooth on \(X\). Moreover, a direct calculation shows that \(\Theta\) is $d$-closed. Note that $\Theta$ is semi-positive.

By $H^{1,1}_{BC}(X,\mathbb{R})=\mathbb{R}$, the first Chern classes of
\(\mathcal O_X(C)\) and \(\mathcal O_X(D)\) in Bott--Chern cohomology
are represented by positive multiples of \(\Theta\).  Hence \(\mathcal O_X(C)\) and \(\mathcal O_X(D)\) are semi-positive.
\end{proof}

For a non-diagonal contraction in normal form, such as $\gamma(z_1, z_2) = (\alpha z_1 + \lambda z_2^m, \beta z_2)$ satisfying the resonance condition $\alpha = \beta^m$ with $\lambda \neq 0$, only the coordinate hyperplane $z_2 = 0$ descends to an invariant elliptic curve $D$ on $X$. The anticanonical divisor is given by
$
-K_X = (m+1)D,
$
where $m$ is the order of the resonance defining the non-diagonal normal form.
Note that \cite[Proposition~6.4]{DPS94} also implies that the anticanonical line bundle of any Hopf surface is nef.
To simplify the notation, (after rescaling the first coordinate) we may assume that $\lambda=1$ in the following.

Now we compute the Ueda type of $(X,D)$ to conclude that this divisor $D$ is not semi-positive.
\begin{proposition}[Descent Data and Gauge Equivalence via Automorphy Factors]
\label{prop:automorphy-factors}
Let $X = \widetilde{X} / \Gamma$ be a complex manifold, where $\Gamma = \pi_1(X, x_0)$ acts freely and properly discontinuously on the universal cover $\widetilde{X}$. Suppose that $\widetilde{X}$ is Stein and contractible $($satisfying $H^1(\widetilde{X}, \mathcal{O}_{\widetilde{X}}) = 0$ and $H^2(\widetilde{X}, \mathbb{Z}) = 0)$. Then:
\begin{enumerate}
    \item The Picard group $\operatorname{Pic}(\widetilde{X})$ is trivial. In particular, any holomorphic line bundle $L \to X$ pulled back to $\widetilde{X}$ is analytically isomorphic to the trivial line bundle $\widetilde{X} \times \mathbb{C}$.
    \item The line bundle $L$ is uniquely determined by lifting the $\Gamma$-action on $\widetilde{X}$ to an action on the total space $\widetilde{X} \times \mathbb{C}$ given by
    \begin{equation}\label{eq:gamma-action}
    \gamma \cdot (x, v) = \bigl(\gamma(x), J(\gamma, x)v\bigr), \quad \text{for } \gamma \in \Gamma, \; x \in \widetilde{X}, \; v \in \mathbb{C},
    \end{equation}
    where $J \colon \Gamma \times \widetilde{X} \to \mathbb{C}^*$ is a holomorphic factor of automorphy satisfying the $1$-cocycle condition
    \begin{equation}\label{eq:1-cocycle}
    J(\gamma_1 \circ \gamma_2, x) = J(\gamma_1, \gamma_2(x)) J(\gamma_2, x) \quad \text{for all } \gamma_1, \gamma_2 \in \Gamma, \; x \in \widetilde{X}.
    \end{equation}
    \item Two factors of automorphy $J_1, J_2 \colon \Gamma \times \widetilde{X} \to \mathbb{C}^*$ define isomorphic holomorphic line bundles over $X$ if and only if they differ by a $\Gamma$-coboundary; that is, there exists a non-vanishing holomorphic function $f \in H^0(\widetilde{X}, \mathcal{O}_{\widetilde{X}}^*)$ such that
    \begin{equation}\label{eq:coboundary-relation}
    J_1(\gamma, x) = J_2(\gamma, x) \cdot \frac{f(\gamma(x))}{f(x)} \quad \text{for all } \gamma \in \Gamma, \; x \in \widetilde{X}.
    \end{equation}
\end{enumerate}
\end{proposition}

\begin{proof}
Since $\widetilde{X}$ is Stein and contractible, $H^1(\widetilde{X}, \mathcal{O}_{\widetilde{X}}) = 0$ and $H^2(\widetilde{X}, \mathbb{Z}) = 0$. The exponential sheaf sequence $0 \to \mathbb{Z} \to \mathcal{O}_{\widetilde{X}} \to \mathcal{O}_{\widetilde{X}}^* \to 0$ yields $\operatorname{Pic}(\widetilde{X}) \cong H^1(\widetilde{X}, \mathcal{O}_{\widetilde{X}}^*) =0$, proving (1). 

Statements (2) and (3) follow directly from the classification of fiber bundles over quotients by properly discontinuous group actions: the associativity of the action $\gamma_1 \cdot (\gamma_2 \cdot (x, v)) = (\gamma_1 \circ \gamma_2) \cdot (x, v)$ yields \eqref{eq:1-cocycle}, while an isomorphism between the bundle structures corresponds to a change of fiber trivialization $v \mapsto f(x)v$, inducing \eqref{eq:coboundary-relation}.
\end{proof}

%Since the non-diagonal Hopf surface is generic among all Hopf surface, the anticanonical line bundle of a generic Hopf surface is not semipositive.
%However, note that  (\cite{DPS94}).

%Choose a branch of $\log\beta$, and define
%\[
%\log\alpha:=m\log\beta.
%\]
%Since
%\[
%e^{m\log\beta}=\beta^m=\alpha,
%\]
%this indeed determines 
Fix a branch of $\log\alpha$. 
%In particular,
%\[
%\frac{\log\alpha}{m}=\log\beta,
%\]
%and therefore
%\[
%e^{\log\alpha/m}=e^{\log\beta}=\beta.
%\]
First, we describe the representations associated with the normal bundle and general flat line bundles on the invariant curve.
The normal bundle $N_{D/X}$ of the invariant elliptic curve $D$ is the flat line bundle associated with the representation
\[
\rho:\pi_1(D)\longrightarrow\mathbb C^*,
\]
given by
\[
\rho(1,0)=1,\qquad
\rho(0,1)=\beta.
\]
Indeed, viewing
\[
D=(\mathbb C^*\times\{0\})/\langle\gamma\rangle,
\]
the universal covering map
\[
\mathbb C\longrightarrow\mathbb C^*,\qquad
\tilde z_1\longmapsto e^{2\pi\sqrt{-1}\tilde z_1},
\]
identifies $\pi_1(D)\cong\mathbb Z^2$, where the generators act on
$\mathbb C$ by
\[
(1,0)\cdot\tilde z_1=\tilde z_1+1,\qquad
(0,1)\cdot\tilde z_1
=\tilde z_1+\frac{\log\alpha}{2\pi\sqrt{-1}}.
\]
The differential of $\gamma$ induces the action
\[
(1,0)\cdot dz_2=dz_2,\qquad
(0,1)\cdot dz_2=\beta\,dz_2
\]
on $N_{D/X}$.
Similarly, we can consider more generally a flat line bundle $L_{\delta}$ associated with the representation
\[
\rho:\pi_1(D)\longrightarrow\mathbb C^*,
\]
given by
\[
\rho(1,0)=1,\qquad
\rho(0,1)=\delta.
\]
Under this notation, $N_{D/X}=L_\beta$.

Next, we recall Siu's neighborhood construction and the lift of the action to the universal cover.

By Siu's theorem \cite[Main Theorem]{Siu76} (see also \cite[Theorem 1]{Dem90}), there exists a Stein open neighborhood $U$ of
$\{z_2=0\}\simeq\C^*$ in $\C^2\setminus\{0\}$ together with a smooth
deformation retraction $U\to\{z_2=0\}$.
Let $U_{\mathrm{univ}}$ denote the pullback of $U$ under the universal
covering map $\C\to\C^*$. Then $U_{\mathrm{univ}}$ is a Stein open
neighborhood of $\C$, and since $\C$ is contractible and
$U_{\mathrm{univ}}$ deformation retracts onto $\C$, it is itself
contractible.

After possibly shrinking $U_{\mathrm{univ}}$, the action of
$\langle\gamma\rangle$ lifts to an action of $\pi_1(D)$ on
$U_{\mathrm{univ}}$ given by
\[
(1,0)\cdot(\tilde z_1,z_2)
=
(\tilde z_1+1,z_2),
\]
and
\[
(0,1)\cdot(\tilde z_1,z_2)
=
\left(
\frac{1}{2\pi\sqrt{-1}}
\log\bigl(\alpha e^{2\pi\sqrt{-1}\tilde z_1}+z_2^m\bigr),
\beta z_2
\right),
\]
where the branch of $\log$ is chosen so that
\[
\log\bigl(\alpha e^{2\pi\sqrt{-1}\tilde z_1}\bigr)
=
\log\alpha+2\pi\sqrt{-1}\tilde z_1.
\]
In particular, we can define the automorphy factors as in Proposition \ref{prop:automorphy-factors}.

Before proceeding to the local calculations, we clarify our strategy for computing the obstruction classes via \v{C}ech cohomology.

The normal bundle is flat, which is equivalent to admitting a flat connection compatible with the holomorphic structure. Since the elliptic curve is compact Kähler, there exists a Hermitian metric on it such that this flat connection is also compatible with the metric. 
In particular, there exists a gauge function that transforms the factor of automorphy into a unitary representation of $\pi_1(D)$. We will extend this local gauge function to the full cover $U_{\mathrm{univ}}$ of the neighborhood.
However, because there is no canonical extension of this flat Hermitian metric to a neighborhood of the elliptic curve, computing the (generalized) Ueda obstruction classes via differential forms is non-trivial. We therefore adapt Ueda's classical approach using \v{C}ech cohomology.

%In this setting, it is more convenient to compute the generalized Ueda obstruction class of $L$ using its \v{C}ech representative, because the factor of automorphy does not naturally define a Hermitian metric.
We record the following remark, which will not be needed in what follows.
\begin{remark}
Since the subvariety $D$ is one-dimensional, we may use the following open cover to compute the  Ueda obstruction class via \v{C}ech cohomology.
Choose any open cover of \(X\), which need not be Stein. In general, the associated \v{C}ech complex does not compute the sheaf cohomology of \(X\). However, the corresponding \v{C}ech cocycle restricts to a \v{C}ech cocycle with respect to the induced open cover of the subvariety, obtained by intersecting each open set with the subvariety. Since the subvariety is one-dimensional, every open subset is Stein. Consequently, the induced cover is a Stein cover, and the corresponding \v{C}ech cocycle computes the sheaf cohomology of the subvariety.
\end{remark}
%However, we observe that since the non-semipositivity criterion is in terms of non-triviality of the generalized Ueda obstruction class, it is enough to show the non-triviality of the corresponding \v{C}ech class with respect to any open cover (not necessarily Stein) since the \v{C}ech c

Let
\[
\tau=\frac{\log\alpha}{2\pi\sqrt{-1}},
\]
where the branch of $\log\alpha$ is fixed as above. Choose
\[
c=\sqrt{-1}\,\frac{\log|\delta|}{\operatorname{Im}\tau}\in\mathbb C.
\]
Then
\[
|e^c|=1,\qquad |e^{c\tau}|=|\delta|^{-1}.
\]
It follows that the gauge transformation
\[
f(\tilde z_1)=e^{c\tilde z_1}
\]
transforms the factor of automorphy $\rho$ into the unitary representation
\[
\rho'(1,0)=e^c,\qquad
\rho'(0,1)=\delta e^{c\tau},
\]
which satisfies
\[
|\rho'(1,0)|=|\rho'(0,1)|=1.
\]
Under the change of fiber coordinate
\(
v=f dz_2,
\)
the factor of automorphy of $L_{\beta}$ is transformed from $\rho$ to the unitary representation $\rho'$.

We extend the holomorphic function
$
f(\tilde z_1)=e^{c\tilde z_1}
$
to $U_{\mathrm{univ}}$ by defining
\[
f(\tilde z_1,z_2)
=
\exp\!\left(
c\left(
\frac{1}{2\pi\sqrt{-1}}
\log\bigl(\alpha e^{2\pi\sqrt{-1}\tilde z_1}+z_2^m\bigr)
-\tau
\right)
\right),
\]
where the branch of $\log$ is chosen as above.

Since \(U_{\mathrm{univ}}\) is Stein and contractible, the pullback of
\(\mathcal O(D)\) to \(U_{\mathrm{univ}}\) is trivial. We identify it with
\(\mathcal O_{U_{\mathrm{univ}}}\) by means of the pullback of the canonical
section of \(\mathcal O(D)\). Using the extended gauge function \(f\), we
define a new factor of automorphy representing the same line bundle
\(\mathcal O(D)\). Restricting to \(z_2=0\), this factor of automorphy reduces
to the unitary representation \(\rho'\) defining
\[
\mathcal O(D)|_D=N_{D/X}.
\]
Take a Stein cover of \(U\) consisting of two open subsets of the form
\[
(A\times\mathbb{C}) \cap U,
\]
where \(A\) is an annulus in the \(z_1\)-plane. The intersection of these two open subsets has two connected components, corresponding to the two possible choices of lifts of the annuli in the covering space. With respect to the natural trivializations induced by these lifts,
the transition function is given by 
$$
\beta  \exp\!\left(
c\left(
\frac{1}{2\pi\sqrt{-1}}
\log\bigl(\alpha e^{2\pi\sqrt{-1}\tilde z_1}+z_2^m\bigr)
-\tau-\tilde{z}_1
\right)
\right)
$$
on one connected component of the intersection of the two Stein open sets, and is equal to 1 on the other connected component.
Similarly, the Ueda obstruction classes can be computed 
as follows.
%from the Taylor expansion of  \[ \partial \left(
%c\left(
%\frac{1}{2\pi\sqrt{-1}}
%\log\bigl(\alpha e^{2\pi\sqrt{-1}\tilde z_1}+z_2^m\bigr)
%-\tau-\tilde{z}_1
%\right)
%\right).\] 
%In particular, the \(i\)-th Ueda obstruction class vanishes for
%\(i<m\), whereas the \(m\)-th Ueda obstruction class is nontrivial.
%(The \v{C}ech cocycle corresponding to the degree-\(m\) term is the constant
%\(
%\frac{c}{2\pi\sqrt{-1}}.
%\))

%Same calculation shows that the $m$-th Ueda obstruction class of $L_\delta$ (see \cite[Section 2.2]{Koi21}) can be computed from the degree $m$ term in the Taylor expansion of  \[ \partial \left(
%c\left(
%\frac{1}{2\pi\sqrt{-1}}
%\log\bigl(\alpha e^{2\pi\sqrt{-1}\tilde z_1}+z_2^m\bigr)
%-\tau-\tilde{z}_1
%\right)
%\right).\]
%Similarly, we have the following result. 
\begin{proposition}\label{prop:ueda-computation-summary}
Let $X$ be a non-diagonal primary Hopf surface, and let $D \subset X$ be the invariant elliptic curve. For any flat line bundle $L_\delta \in \operatorname{Pic}(X)$ with $|\delta| \neq 1$, its Ueda obstruction classes $u_i(L_\delta) \in H^1(D, \mathcal{O}_D)$ satisfy:
\begin{enumerate}
    \item $u_i(L_\delta) = 0$ for all $1 \le i < m$.
    \item The $m$-th Ueda obstruction class $u_m(L_\delta)$ is non-trivial.
\end{enumerate}
In particular, for any $|\delta| \neq 1$, the flat line bundle $L_\delta$ (and specifically $\mathcal{O}_X(D) = L_\beta$) fails to admit any smooth semi-positive Hermitian metric.
\end{proposition}

\begin{proof}
The transition function on the two-set annulus cover of the Siu neighborhood $U$ is given by
\[
g_{12}(\tilde{z}_1, z_2) = \delta \exp\left( c\left( \frac{1}{2\pi\sqrt{-1}} \log\bigl(\alpha e^{2\pi\sqrt{-1}\tilde z_1} + z_2^m\bigr) - \tau - \tilde{z}_1 \right) \right)
\]
on one connected component of the intersection of the two Stein open sets, and is equal to 1 on the other connected component.
Following Koike's formulation \cite[Section~2.2]{Koi21}, up to a non-zero multiple, the $i$-th Ueda obstruction class corresponds to the degree-$i$ coefficient in $z_2$ of the Taylor expansion of $\partial \log g_{12}$. Differentiating the logarithmic term with respect to $\tilde{z}_1$ shows that all terms of order $i < m$ vanish, whereas the degree-$m$ term yields the non-zero constant class. 
More precisely, on one connected component of the intersection of the two Stein open sets, we have
\[
    \log g_{12}(\tilde{z}_1, z_2) = \log\delta + \frac{c}{2\pi\sqrt{-1}} \cdot \frac{z_2^m}{\alpha e^{2\pi\sqrt{-1}\tilde{z}_1}} + \mathcal{O}(z_2^{2m}),
    \]
whereas $ \log g_{12}(\tilde{z}_1, z_2)=0$  on the other connected component. 
Therefore, the coefficient of $\frac{z_2^m}{\alpha}$
 defines the \v{C}ech
$1$-cocycle 
$$\frac{c}{2\pi\sqrt{-1}} \cdot \frac{1}{ e^{2\pi\sqrt{-1}\tilde{z}_1}}=\frac{c}{2\pi\sqrt{-1}} \cdot \frac{1}{ z_1} $$ 
on one connected component of the overlap, and 0 on the other. 
(Since $|\delta| \neq 1$, $c \neq 0$.)
This cocycle represents a non-trivial element of $H^1\big(D, \mathcal{O}_D(N_{D/X}^{*\otimes m})\big) \cong H^1(D, \mathcal{O}_D)$.
Indeed, suppose that this \v{C}ech $1$-cocycle were a coboundary. Then there
would exist holomorphic functions $f_1$ and $f_2$ on the two Stein open sets
whose \v{C}ech differential gives the above cocycle. Considering the Laurent
expansions of
\[
f_1(z_1)-f_2(z_1)=0
\]
and
\[
f_1(z_1)-f_2(\alpha z_1)=\frac{c}{2\pi\sqrt{-1}} \cdot \frac{1}{ z_1}
\]
on the two connected components of the intersection, we obtain a
contradiction by the uniqueness of Laurent expansions.

The non-semi-positivity then follows directly from Koike's criterion \cite[Theorem~1.2]{Koi22}.
\end{proof}
Now we investigate the relationship between the first Chern classes of the flat line bundles $L_\delta$.
\begin{definition}
Let $X$ be a compact complex manifold. The Bott--Chern cohomology group of
bidegree $(1,1)$ is defined by
\[
H^{1,1}_{\mathrm{BC}}(X,\mathbb{C})
:=
\frac{
\{\alpha\in A^{1,1}(X)\mid d\alpha=0\}
}{
\sqrt{-1}\partial\overline{\partial} A^{0,0}(X)
}.
\]
The real Bott--Chern cohomology group is the real subspace
\[
H^{1,1}_{\mathrm{BC}}(X,\mathbb{R})
:=
\left\{
[\alpha]\in H^{1,1}_{\mathrm{BC}}(X,\mathbb{C})
\mid
\overline{[\alpha]}=[\alpha]
\right\}.
\]
Equivalently,
\[
H^{1,1}_{\mathrm{BC}}(X,\mathbb{R})
=
\frac{
\{\alpha\in A^{1,1}(X,\mathbb{R})\mid d\alpha=0\}
}{
\sqrt{-1}\partial\overline{\partial} C^\infty(X,\mathbb{R})
}.
\]
\end{definition}
Recall that for any line bundle $L \in \operatorname{Pic}(X)$, its Bott--Chern first Chern class $c_1(L) \in H^{1,1}_{\mathrm{BC}}(X, \mathbb{R})$ is represented by the Chern curvature form $\frac{\sqrt{-1}}{2\pi} \Theta(L, h) = -\frac{\sqrt{-1}}{2\pi} \partial \bar{\partial} \log h$ associated with any smooth Hermitian metric $h$ on $L$.
\begin{proposition}[Bott--Chern Chern Class Relation for Flat Line Bundles]
\label{prop:chern-relation}
Let $X$ be a primary Hopf manifold of dimension $n$, and for each $\delta \in \mathbb{C}^*$, let $L_\delta$ denote the flat line bundle associated with the representation of $\pi_1(X) \cong \mathbb{Z}$ mapping $1 \mapsto \delta$. Then for any $\delta_1, \delta_2 \in \mathbb{C}^*$, their Bott--Chern first Chern classes satisfy the proportionality formula
\begin{equation}\label{eq:chern-relation}
\log|\delta_2|\,c_1(L_{\delta_1}) = \log|\delta_1|\,c_1(L_{\delta_2}) \quad \text{in } H^{1,1}_{\mathrm{BC}}(X,\mathbb{R}).
\end{equation}
In particular, $c_1(L_\delta)$ depends solely on the modulus $|\delta|$.
\end{proposition}

\begin{proof}
By a theorem of Mall \cite[Theorem~4]{Mall91}, every holomorphic line bundle over a Hopf manifold $X$ is flat, which yields an isomorphism $\operatorname{Pic}(X) \cong H^1(X, \mathbb{C}^*) \cong \mathbb{C}^*$. Consider the first Chern class map
\[
c_1 \colon \operatorname{Pic}(X) \longrightarrow H^{1,1}_{\mathrm{BC}}(X,\mathbb{R}).
\]
Since $X$ is diffeomorphic to $S^1 \times S^{2n-1}$, we can project $X$ onto $S^1$. Cover $S^1$ by two open intervals $U_1, U_2 \subset S^1$ whose overlap $U_1 \cap U_2$ consists of two disjoint connected components. Pulling these back yields an open cover $\{X_1, X_2\}$ of $X$, where $X_i = \pi^{-1}(U_i)$. On this cover, any flat line bundle $L_\delta \in \operatorname{Pic}(X)$ is trivialized such that its transition function on one component of $X_1 \cap X_2$ is $1$ and on the other component is the constant $\delta \in \mathbb{C}^*$. 

Fixing local flat metrics $h_1, h_2$ on $X_1$ and $X_2$, and gluing them via a partition of unity $\{\rho_1, \rho_2\}$ subordinate to $\{X_1, X_2\}$ constructs a smooth Hermitian metric $h_\delta$ on $L_\delta$. The local expression of $h_\delta$ depends smoothly on $\delta \in \mathbb{C}^*$. Thus, the curvature forms $\Theta(L_\delta, h_\delta) = -\partial\bar{\partial}\log h_\delta$ vary smoothly in $\delta$, implying that the first Chern class map
\[
c_1 \colon \operatorname{Pic}(X) \longrightarrow H^{1,1}_{\mathrm{BC}}(X, \mathbb{R})
\]
induces a smooth Lie group homomorphism, where $H^{1,1}_{BC}(X,\C)$ is endowed with the quotient topology induced from the Fréchet topology on the space of smooth $(1,1)$-forms.

The kernel of $c_1$ consists of unitary flat line bundles, corresponding to $\mathbb{S}^1 \subset \mathbb{C}^*$. Furthermore, $c_1$ is non-trivial: for instance, the line bundle associated with an invariant divisor $D$ satisfies $\int_X c_1(\mathcal{O}_X(D)) \wedge \omega^{n-1} > 0$ for any Gauduchon metric $\omega$, so $c_1(\mathcal{O}_X(D)) \neq 0$.

It follows that $c_1$ factors through a non-trivial Lie group homomorphism on the quotient:
\[
\operatorname{Pic}(X)/\mathbb{S}^1 \cong \mathbb{C}^*/\mathbb{S}^1 \cong \mathbb{R}_{>0} \longrightarrow H^{1,1}_{\mathrm{BC}}(X,\mathbb{R}).
\]
Since $\mathbb{R}_{>0}$ (under multiplication) is a $1$-dimensional connected Lie group, its image under $c_1$ is a $1$-dimensional real subspace in $H^{1,1}_{\mathrm{BC}}(X,\mathbb{R})$. Every continuous Lie group homomorphism from $(\mathbb{R}_{>0}, \times)$ into the additive Lie group underlying the real vector space is uniquely of the form $|\delta| \mapsto \log|\delta| \cdot v$ for some fixed non-zero vector $v \in H^{1,1}_{\mathrm{BC}}(X,\mathbb{R})$. Evaluating this map at $\delta_1$ and $\delta_2$ immediately yields Equation~\eqref{eq:chern-relation}.
\end{proof}
For completeness, we briefly recall the relevant positivity cones on a compact complex manifold (compare \cite[Definition 6.16]{Dem12}).
\begin{definition}
\label{cone-def}
Let $X$ be a compact complex manifold with a Hermitian metric $\omega$.
\begin{enumerate}
\item A class $\alpha\in H^{1,1}_{\mathrm{BC}}(X,\mathbb{R})$ is called nef if,
for every $\varepsilon>0$, there exists a smooth representative
$\alpha_\varepsilon\in\alpha$ such that
\[
\alpha_\varepsilon\geq -\varepsilon\omega .
\]
The nef cone is defined by
\[
\operatorname{Nef}(X)
:=
\left\{
\alpha\in H^{1,1}_{\mathrm{BC}}(X,\mathbb{R})
\mid
\alpha\ \text{is nef}
\right\}.
\]
\item A class $\alpha\in H^{1,1}_{\mathrm{BC}}(X,\mathbb{R})$ is called
pseudo-effective if there exists a closed positive $(1,1)$-current
$T$ such that
\[
[T]=\alpha .
\]
The pseudo-effective cone is
\[
\operatorname{Psef}(X)
:=
\left\{
\alpha\in H^{1,1}_{\mathrm{BC}}(X,\mathbb{R})
\mid
\alpha\ \text{is pseudo-effective}
\right\}.
\]
\item The effective cone is defined by
\[
\operatorname{Eff}(X)
:=
\left\{
\alpha\in H^{1,1}_{\mathrm{BC}}(X,\mathbb{R})
\mid
\alpha\ \text{is represented by an}\, \mathbb{R} \text{-effective divisor}
\right\}.
\]
\item The semi-positive cone is defined by
\[
\operatorname{SemiPos}(X)
:=
\left\{
\alpha\in H^{1,1}_{\mathrm{BC}}(X,\mathbb{R})
\mid
\alpha\ \text{admits a smooth semi-positive representative}
\right\}.
\]
\end{enumerate}
\end{definition}
By definition, we have
\[
\operatorname{SemiPos}(X)
\subset
\operatorname{Nef}(X)
\subset
\operatorname{Psef}(X).
\]

Recall that $H^{1,1}_{\mathrm{BC}}(X,\mathbb{R})$ is one-dimensional for every primary Hopf surface $X$ (see e.g. \cite[Theorem~4.2]{IO25}), and that $X$ always contains a nef elliptic curve $C$ with non-trivial first Chern class $c_1(\mathcal{O}_X(C)) \neq 0$.

By dimensional considerations, the nef cone $\operatorname{Nef}(X)$ coincides with the pseudo-effective cone $\operatorname{Psef}(X)$, both being generated by $c_1(\mathcal{O}_X(C))$ and thus isomorphic to a half-line $\mathbb{R}_{\ge 0}$. Indeed, $\operatorname{Psef}(X)$ cannot contain a full line since the pairing $\int_X T \wedge \omega > 0$ with any Gauduchon metric $\omega$ is strictly positive for non-zero positive currents $T$.

Furthermore, Proposition~\ref{prop:chern-relation} implies that the Bott--Chern first Chern class of any flat line bundle $L_\delta$ is a non-negative multiple of $c_1(\mathcal{O}_X(C))$ if and only if $|\delta| \le 1$. 
Since $c_1(L_\delta)$ is represented by an effective \(\mathbb R\)-divisor whenever $|\delta| \le 1$, every nef class is effective.
We thus obtain the explicit characterization
\[
\operatorname{Nef}(X) =\operatorname{Eff}(X)= \operatorname{Psef}(X) = \left\{ c_1(L_\delta) \;\middle|\; |\delta| \leq 1 \right\}.
\]

Regarding smooth semi-positivity, the structure depends fundamentally on whether the surface is diagonal or non-diagonal:

\begin{theorem}[Positivity Cones on Primary Hopf Surfaces]
\label{thm:hopf-surface-cones}
Let $X$ be a primary Hopf surface.
\begin{enumerate}
    \item $\operatorname{Nef}(X) =\operatorname{Eff}(X)= \operatorname{Psef}(X) \cong \mathbb{R}_{\ge 0}$, explicitly given by $\{ c_1(L_\delta) \mid |\delta| \le 1 \}$.
    \item If $X$ is diagonal, then every nef line bundle admits a smooth semi-positive Hermitian metric, so
    \[
    \operatorname{SemiPos}(X) = \operatorname{Nef}(X)=\operatorname{Eff}(X) = \operatorname{Psef}(X).
    \]
    \item If $X$ is non-diagonal, the Ueda obstruction class of $D$ prevents $\mathcal{O}_X(D)$ from admitting a smooth semi-positive metric, whence
    \[
    \operatorname{SemiPos}(X) = \{0\}.
    \]
\end{enumerate}
\end{theorem}

%Recall that $H^{1,1}_{\mathrm{BC}}(X,\mathbb{R})$ is one-dimensional for every primary Hopf surface, and that there always exists a nef elliptic curve whose first Chern class is non-trivial. 
%By dimensional considerations, the nef cone coincides with the pseudo-effective cone, and both are generated by the class of this elliptic curve; in particular, they are isomorphic to a half-line.
%(Note that the pseudo-effective cone cannot contain a whole line, as can be seen by intersecting with a Gauduchon metric.)

%On the other hand, by Equation~\eqref{eq:chern-relation}, the first Chern class of any line bundle is a multiple of the first Chern class of the above elliptic curve.
%Consequently, we obtain an explicit description of the pseudo-effective cone as
%\[
%\left\{ c_1(L_\delta)\mid |\delta|\leq 1\right\}.
%\]
%Therefore, the intersection of the Néron--Severi cone with the nef cone is generated by the class of this elliptic curve and is isomorphic to a half-line.

%Since there exists a nef elliptic curve which is not semipositive
%on a non-diagonal Hopf surface, while there exists a non-Hermitian flat
%semipositive elliptic curve on a diagonal Hopf surface. Therefore, for a
%diagonal Hopf surface,
%\[
%\operatorname{SemiPos}(X)=\operatorname{Nef}(X),
%\]
%whereas for a non-diagonal Hopf surface,
%\[
%\operatorname{SemiPos}(X)=\{0\}.
%\]

\section{Positive cones on Hopf manifolds}
We recall the following normal form for contracting automorphisms due to Reich
\cite{Rei,Reib}. 
This normal form can be applied to the contraction defining a Hopf
manifold.

\begin{lemma}[Normal form of Hopf manifolds]
\label{normal-form}
Let \(X=(\mathbb{C}^n\setminus\{0\})/\langle\gamma\rangle\) be a Hopf manifold,
where \(\gamma\) is a contraction. Then, after a holomorphic change of coordinates,
the generator \(\gamma\) can be written in the following form:
\begin{equation}\label{eq:1}
\begin{aligned}
z_1' &= \alpha_1 z_1,\\
z_2' &= z_1+\alpha_2 z_2,\\
&\ \vdots\\
z_{r_1}' &= z_{r_1-1}+\alpha_{r_1}z_{r_1},\\
z_{r_1+1}' &= \alpha_{r_1+1}z_{r_1+1}
+P_{r_1+1}(z_1,\ldots,z_{r_1}),\\
&\ \vdots\\
z_{r_1+r_2}' &= z_{r_1+r_2-1}
+\alpha_{r_1+r_2}z_{r_1+r_2}
+P_{r_1+r_2}(z_1,\ldots,z_{r_1}),\\
z_{r_1+r_2+1}' &=
\alpha_{r_1+r_2+1}z_{r_1+r_2+1}
+P_{r_1+r_2+1}(z_1,\ldots,z_{r_1+r_2}),\\
&\ \vdots\\
z_n' &=
z_{n-1}
+\alpha_n z_n
+P_n(z_1,\ldots,z_{r_1+\cdots+r_{\mu-1}}).
\end{aligned}
\tag{1}
\end{equation}
Here
\[
1>|\alpha_1|\geq\cdots\geq|\alpha_n|>0,
\]
and \(\mu\) denotes the number of Jordan blocks of the linear part of
\(\gamma\). For
\[
r_1+\cdots+r_s<j\leq r_1+\cdots+r_{s+1},
\]
the polynomials \(P_j\) are finite sums of monomials
\(
z_1^{m_1}\cdots z_{r_s}^{m_{r_s}}
\)
satisfying the resonance condition
\[
\alpha_j=
\alpha_1^{m_1}\cdots\alpha_{r_s}^{m_{r_s}},
\qquad
m_1+\cdots+m_{r_s}\geq2.
\]
Moreover, some of the polynomials \(P_j\) may be identically zero. Since the
linear part is written in Jordan normal form, the eigenvalues are constant on
each Jordan block; in particular,
\[
\alpha_1=\cdots=\alpha_{r_1},
\]
and similarly for the remaining Jordan blocks.
\end{lemma}
The results in \cite{Rei, Reib} establish a formal and analytic classification of biholomorphic contractions with an attracting fixed point in arbitrary dimensions. In the context of Hopf manifolds, Reich's normal form serves as the natural higher-dimensional analog of the classical Poincaré--Dulac normal form for surfaces, systematically accounting for resonant non-diagonal linear parts and Jordan blocks in dimension $n \ge 2$.

We first establish the following restriction property for the Bott--Chern cohomology of primary Hopf manifolds.
The existence of this flag may be viewed as a higher-dimensional generalization of Kodaira's result \cite{Koi66} on the existence of an elliptic curve on a primary Hopf surface.
\begin{lemma}
\label{bc-isom-lemma}
Let \(X\) be a primary Hopf manifold of dimension \(n\) with normal form as in Lemma \ref{normal-form}, and let
$
Z_{n-2}\subset X
$
be the Hopf surface in the natural flag
\[
X=Z_0\supset Z_1\supset\cdots\supset Z_{n-1},
\]
where
\[
Z_i:=\{z_1=\cdots=z_i=0\}.
\]
Then the restriction map
\[
\iota^*:H^{1,1}_{BC}(X,\mathbb{R})
\longrightarrow
H^{1,1}_{BC}(Z_{n-2},\mathbb{R})
\]
induced by the inclusion
\(\iota:Z_{n-2}\hookrightarrow X\) is an isomorphism.
\end{lemma}
\begin{proof}
By \cite[Theorem~4.2]{IO25}, together with
\cite[Theorem~3.1, Corollaries~5.2 and~5.4]{IO23}, we have
\[
\dim_{\mathbb R}H^{1,1}_{BC}(X,\mathbb R)=1
\]
for every primary Hopf manifold. Since \(Z_{n-2}\) is a primary Hopf surface, the
same result gives
\[
\dim_{\mathbb R}H^{1,1}_{BC}(Z_{n-2},\mathbb R)=1.
\]

It remains to show that \(\iota^*\) is nonzero. The submanifold
\(Z_{n-1}\subset Z_{n-2}\) is an elliptic curve, and the line bundle
\(
\mathcal O_{Z_{n-2}}(Z_{n-1})
\)
is nef and flat. Hence it is defined by a character
\[
\rho:\pi_1(Z_{n-2})\longrightarrow\mathbb C^* .
\]
The inclusion \(Z_{n-2}\hookrightarrow X\) induces an identification
\[
\pi_1(Z_{n-2})\simeq \pi_1(X)\simeq\mathbb Z .
\]
Therefore the same character defines a flat line bundle \(L\) on \(X\).
By construction,
\[
\iota^*c_1(L)
=
c_1(\mathcal O_{Z_{n-2}}(Z_{n-1})).
\]
The class on the right-hand side is nonzero, and hence
\[
\iota^*:H^{1,1}_{BC}(X,\mathbb R)
\longrightarrow
H^{1,1}_{BC}(Z_{n-2},\mathbb R)
\]
is nonzero. Since both vector spaces are one-dimensional, \(\iota^*\) is an
isomorphism.
\end{proof}

The following proposition describes the nef, pseudo-effective, and semi-positive cones of
a primary Hopf manifold in terms of those of the Hopf surface appearing in its natural
flag.
\begin{proposition}
\label{cone-hopf-mfd}
Let \(X\) be a primary Hopf manifold with normal form as in Lemma \ref{normal-form}. Then the restriction map
\[
\iota^*:H^{1,1}_{\mathrm{BC}}(X,\mathbb{R})
\longrightarrow
H^{1,1}_{\mathrm{BC}}(Z_{n-2},\mathbb{R})
\]
identifies the nef cone, the effective cone and the pseudo-effective cone of \(X\) with those of the
Hopf surface \(Z_{n-2}\). Moreover, 
the semi-positive cone of \(X\) is determined by whether \(Z_{n-2}\) is a
diagonal or non-diagonal Hopf surface.
\end{proposition}

\begin{proof}
Since the restriction of a semi-positive representative remains semi-positive,
and the restriction of a nef class remains nef, the above isomorphism in Lemma \ref{bc-isom-lemma} is
compatible with the nef and semi-positive cones.

By \cite[Theorem~4.2]{IO25}, the Bott--Chern cohomology group
\(
H^{1,1}_{\mathrm{BC}}(X,\mathbb{R})
\)
is one-dimensional.
(In particular, as one-dimensional real vector spaces, the relevant nonzero cones are half-lines, and preservation of nefness/semi-positivity determines the orientation.)
Furthermore, the pseudo-effective cone cannot contain a
line, since the intersection with a Gauduchon metric is nonnegative on the
pseudo-effective cone and strictly positive for any nonzero pseudo-effective
class.
Therefore, in the one-dimensional space
\(H^{1,1}_{\mathrm{BC}}(X,\mathbb{R})\), the pseudo-effective cone coincides with
the nef cone.
By Proposition \ref{prop:chern-relation}, any pseudo-effective class contains some non-negative multiple of $[Z_1]$.

Consequently, the nef and semi-positive cones of \(X\) are completely determined
by the corresponding cones on the Hopf surface \(Z_{n-2}\). The latter are
known according to whether \(Z_{n-2}\) is diagonal or non-diagonal. This gives
the desired characterization of nef line bundles on \(X\) which do not admit
semi-positive representatives.
\end{proof}

We recall the following characterization of nef classes due to P\u{a}un \cite[Theorem~1.C.2]{Paun98}.
We emphasize that the proof is constructive.
\begin{theorem}
\label{paun-thm}
Let $T$ be a closed positive $(1,1)$-current on a compact complex
manifold $X$. Then the Bott--Chern cohomology class $\{T\}$ is nef if and
only if the restriction $\{T\}|_Z$ is nef for every irreducible analytic
subset
\[
Z\subset \bigcup_{c>0} E_c(T)
\]
where \[
E_c(T):=\{x\in X\mid \nu(T,x)\ge c\},
\]
and $\nu(T,x)$ denotes the Lelong number of $T$ at $x$.

In particular, if $\bigcup_{c>0} E_c(T)$ is a smooth submanifold, $\{T\}$ is nef if and
only if the restriction $\{T\}|_{E_c(T)}$ is nef. 
\end{theorem}
If one is only concerned with the Néron--Severi part of the nef cone,
without appealing to the computation of
$\dim H^{1,1}_{\mathrm{BC}}(X,\mathbb{R})$, then it can be described by
means of Theorem~\ref{paun-thm}.
\begin{remark}
\label{rem: Paun}
Let $\alpha_i$ denote the constant corresponding to the flat line bundle
$\mathcal{O}_{Z_i}(Z_{i+1})$ for $0\leq i\leq n-2$, where the flat line
bundles on the submanifolds $Z_i$ are identified via the isomorphisms of
their fundamental groups induced by the inclusions. By construction,
$|\alpha_i|<1$ for every $i$.

By Equation~\eqref{eq:chern-relation}, the first Chern classes
$c_1(L_{\alpha_i})$ differ by positive scalar multiples. Consider the
line bundle $\mathcal{O}_X(Z_1)$, whose first Chern class contains the
current of integration along $Z_1$. By Theorem~\ref{paun-thm} and
induction on the dimension, $\mathcal{O}_X(Z_1)$ is nef, since
$\mathcal{O}_{Z_1}(Z_2)$ is nef by the induction hypothesis.
Consequently, Equation~\eqref{eq:chern-relation} shows that
$\mathcal{O}_X(Z_1)$ generates the intersection of the nef cone and the image of the Picard group in Bott--Chern cohomology.
\end{remark}

We recall the following classical lemma and include a brief proof for the convenience of the reader.
\begin{lemma}\label{lem:BC-Galois}
Let $\pi:Y\to X$ be a finite \'etale Galois cover with Galois group $G$.
Then the pull-back induces an isomorphism
\[
\pi^*:
H^{p,q}_{\mathrm{BC}}(X,\mathbb{C})
\xrightarrow{\;\cong\;}
H^{p,q}_{\mathrm{BC}}(Y,\mathbb{C})^G,
\]
where $H^{p,q}_{\mathrm{BC}}(Y,\mathbb{C})^G$ denotes the $G$-invariant subspace.
\end{lemma}

\begin{proof}
Since $\pi$ is finite and \'etale, the pull-back and push-forward of differential forms commute with $\partial$ and $\bar\partial$, and hence induce morphisms
\[
\pi^*:H^{p,q}_{\mathrm{BC}}(X,\mathbb{C})
\longrightarrow
H^{p,q}_{\mathrm{BC}}(Y,\mathbb{C}),
\qquad
\pi_*:H^{p,q}_{\mathrm{BC}}(Y,\mathbb{C})
\longrightarrow
H^{p,q}_{\mathrm{BC}}(X,\mathbb{C}).
\]
Moreover,
\[
\pi_*\pi^*=(\deg\pi)\,\mathrm{id},
\qquad
\pi^*\pi_*=\sum_{g\in G}g^*.
\]
It follows that $\pi^*$ is injective and its image is precisely
$H^{p,q}_{\mathrm{BC}}(Y,\mathbb{C})^G$.
\end{proof}
Note that the cones defined in Definition~\ref{cone-def} are preserved under pull-back and push-forward under finite \'etale Galois covers.

Now we discuss secondary Hopf manifolds.

\begin{remark}
\label{rem: secondary-Hopf}
By definition, a secondary Hopf manifold is a finite \'etale Galois quotient of a primary Hopf manifold.
By Lemma~\ref{lem:BC-Galois}, the Bott--Chern cohomology group
$
H^{1,1}_{\mathrm{BC}}(X,\mathbb{C})
$
is either trivial or one-dimensional.
In the former case, all the cones in Definition~\ref{cone-def} are trivial.
In the latter case, the descriptions of the nef, pseudo-effective, effective, and semi-positive cones are identical to those for primary Hopf manifolds.

Moreover, the semi-positive cone is determined by whether a (equivalently, any) finite \'etale Galois cover by a primary Hopf manifold contains a diagonal Hopf surface as a submanifold.
Indeed, this condition is independent of the choice of the finite \'etale Galois cover.
\end{remark}
For the convenience of the reader, we give a simpler proof of 
\cite[Article 2, Proposition 9.1]{Nee} following the ideas from 
\cite[Theorem 5.1]{KP90}.

\begin{proposition}
\label{prop-class}
Let $X$ be a minimal compact non-algebraic surface.
Assume that there exists an elliptic curve $D$ on $X$ such that some Ueda class is non-trivial.
Then $X$ is a Hopf surface.
\end{proposition}

\begin{proof}
We prove this by considering the algebraic dimension $a(X)$ of $X$.
The non-triviality of some Ueda obstruction class implies by 
\cite[Theorem 1]{Ued83} that $X\setminus D$ is strongly $1$-convex.

Assume first that $a(X)=1$. By the Kodaira--Enriques classification, 
$X$ admits an elliptic fibration
$
f:X\to C.
$
Note that
$
D^2=0
$
since the normal bundle of $D$ is Hermitian flat.
If $D$ is a multisection of $f$, then
\[
(D+F)^2=2(D\cdot F)>0,
\]
where $F$ is a general fiber of $f$.
By \cite[Theorem 8]{Koi66}, it would then be projective, a contradiction.
If $D$ is contained in a fiber of $f$, then $f$ induces a proper fibration from
$X\setminus D$ to an open Riemann surface.
This contradicts the strong $1$-convexity of $X\setminus D$.

Now assume that $a(X)=0$. By the Kodaira--Enriques classification, if $X$ is K\"ahler, 
then $X$ is either a torus or a K3 surface.
However, a torus or a K3 surface of algebraic dimension zero contains no elliptic curve.
Hence $X$ is non-K\"ahler and therefore belongs to class $VII_0$.

Suppose that $b_2(X)>0$, by \cite[Proposition 4.12]{Eno81} or \cite[Theorem 10.2]{Nak84}, $X$ is a parabolic Inoue surface. 
However a parabolic Inoue surface contains a unique elliptic curve $C$ and the self-intersection of $C$ satisfies $(C^2)<0$.
Thus $X$ cannot be a parabolic Inoue surface.
%By \cite[Main Theorem]{Eno81}, 
%$X\setminus D$ is biholomorphic to the complement of the zero section $C_0$
%in a projective bundle over an elliptic curve, where
%\[
%C_0^2=-b_2(X).
%\]
%Thus $C_0$ can be blown down by Grauert's criterion to a normal surface $Y$
%such that $C_0$ is mapped to a point $p$.
%By \cite[Theorem 1]{Ued83}, there exists a non-constant psh function on
%$X\setminus D$.
%The induced psh function on $Y\setminus\{p\}$ extends across $p$ since $Y$
%is normal.
%Since $Y$ is compact, the extended psh function is constant, a contradiction.

Therefore,
$
b_2(X)=0.
$
By \cite[Theorem 34]{Koi66}, $X$ is a Hopf surface since it contains a curve.
\end{proof}
%Define $D_1=\{z_1=0\}$.
%Define
%\[
%R_s:=r_1+\cdots+r_s,\qquad R_0:=0.
%\]
%and
%\[
%I:=
%\left\{
%1\le s\le\mu-1\;\middle|\;
%P_{R_s+1}
%=
%\cdots
%=
%P_{R_{s+1}}
%=
%0
%\right\}.
%\]
%Define \[
%D_{s+1}
%=
%\{z_{R_s+1}=0\},
%\qquad s\in I.
%\]
%Thus
%$$-K_X=\sum_{s \in I} r_{s+1} D_{s+1} +(\sum_{s \notin I }r_{s+1}+r_1)D_1.$$
\section{Appendix: Classification of Nef Line Bundles on Compact Surfaces with Finite Ueda Type}
The main purpose of this section is to show the following theorem.
\begin{theorem}
Let $X$ be a minimal compact complex surface. Assume that $X$ contains
an elliptic curve $C$ and a nef line bundle $L$ on $X$ such that the Ueda type
of $(X,C,L)$ is finite (see Definition~\ref{ueda-def}). Then $X$ is either
Serre's example (see e.g. \cite[Example 1.7]{DPS94}) or a Hopf surface. 
\end{theorem}

Let us first compare our theorem with the corresponding results of
Neeman \cite[Article 1, Section 6; Article 2, Section 9]{Nee}.
\begin{remark}
By \cite[Remark 3.3]{Koi21}, the Ueda type of
$(X,C,\mathcal{O}_X(C))$ coincides with the Ueda type of $(X,C)$.
Neeman classified minimal compact complex surfaces $X$ containing an
elliptic curve $C$ for which the Ueda type of $(X,C)$ is finite.
The resulting classification is the same as ours.

The arguments, however, are quite different. Neeman's proof relies
essentially on \cite[Theorem 1]{Ued83}, which implies that
$X\setminus C$ is strongly $1$-convex. This imposes strong topological
constraints on $X$, and, together with the Kodaira--Enriques
classification of compact complex surfaces available at that time,
allows one to obtain the classification.

In our setting, there is a priori no relation between \(L\) and \(C\). Consequently, Ueda's theorem does not directly
provide the same topological information about $X\setminus C$.
We therefore need to use the finer classification theory of compact
complex surfaces in order to obtain the classification.
\end{remark}
For the convenience of the reader, we recall the construction of Ueda type of $(X,C,L)$ given in \cite[Section 3.2]{Koi21}.
\begin{definition}
\label{ueda-def}
Let $Y$ be a smooth hypersurface in a complex manifold $X$, and let
$L$ be a line bundle on $X$ such that $L|_Y$ is unitary flat. Take a sufficiently fine
open covering $\{V_j\}$ of a neighborhood of $Y$, with
$U_j=V_j\cap Y$, and choose local defining functions $w_j$ for $Y$
on $V_j$. Let $e_j$ be local frames of $L$ such that the restriction of $L$
to $U_j$ is unitary flat, with constant transition functions
$t_{jk}\in U(1)$.

Assume that
after choosing the defining functions and frames so that the system
is of type $n$, i.e. we can write
$t_{jk}^{-1}e_k/e_j=1+f_{kj,n}w_j^n+O(w_j^{n+1})$ on $V_j\cap V_k$.
The coefficients $f_{kj,n}$ define a \v{C}ech $1$-cocycle with values
in $N_{Y/X}^{* \otimes n}$.
The $n$-th obstruction class is defined by
$u_n(Y,X,L)=[\{f_{kj,n}\}_{j,k}]\in H^1(Y,N_{Y/X}^{*\otimes n})$.

More precisely, identifying $N_{Y/X}^*$ with the conormal bundle
generated locally by $dw_j$, the class is represented by
$\{f_{kj,n}(dw_j)^{\otimes n}\}_{j,k}$. The class is independent of
the choice of a system of type $n$ if one of the following conditions is satisfied:
\begin{enumerate}
\item $n=1$;
\item $N_{Y/X}$ is unitary flat and the system $\{(V_j,w_j)\}$ of
local defining functions of $Y$ is of type $n$;
\item $H^0(Y,N_{Y/X}^{* \otimes m})=0$ for every integer $m$ satisfying
$1\leq m\leq n-1$.
\end{enumerate}
The Ueda type of $(Y,X,L)$ is defined to be 
%$\infty$ if
%$u_n(Y,X,L)=0$ for every $n\geq1$; otherwise it is the smallest
finite if there exists a
positive integer $n$ for which $u_n(Y,X,L)\neq0$ for some choice of system of local defining functions and local frames.
\end{definition}
\begin{remark}
\label{rem-first-chern}
Suppose that $X$ is compact. If two line bundles $L_1$ and $L_2$ on
$X$ have the same first Chern class in Bott--Chern cohomology, then
$L_1\otimes L_2^{-1}$ is unitary flat. Hence, after choosing a
unitary-flat system of local frames for $L_1\otimes L_2^{-1}$, the
construction above gives the same obstruction classes for $L_1$ and
$L_2$. In particular, one may define the Ueda type associated with a
Bott--Chern $(1,1)$-class rather than with a particular line bundle
representing that class.
\end{remark}
\begin{remark}
If $(C^2)> 0$,
by \cite[Theorem 8]{Koi66}, $X$ is projective.
Since $L$ is nef and $L|_C$ is unitary flat, $(L^2) \geq 0$,$(L \cdot C)=0$.
By Hodge index theorem (see e.g. \cite[Chapter IV, Theorem 2.14]{BHPV}),
either $L$ is unitary flat or $L$ is big and nef.
In the first case, by Definition \ref{ueda-def}, the Ueda type of $(X, C, L)$ is not finite.
In the latter case, by Kodaira's lemma, $L=A+E$ for some ample $\mathbb{Q}$-divisor $A$ and effective $\mathbb{Q}$-divisor $E$.
Thus $(L \cdot C)=(A \cdot C)+(E \cdot C)>0$
since $C$ is nef.
This yields a contradiction.

Thus in the following, we will always assume that $(C^2)=0$.
\end{remark}

Since the Ueda type is finite, $(C^2)=0$  and $C$ is an elliptic curve, the normal
bundle $N_{C/X}$ is torsion.
In order for the Ueda class to be well-defined, we need
to assume that $L|_C$ is unitary flat. In particular, $(L \cdot C)=0$.

Consider first the case that $X$ is K\"ahler.
Since $L$ is nef, we have $(L^2)\geq 0.$ 
If $(L^2)> 0$,
by \cite[Theorem 8]{Koi66}, $X$ is projective.
Since $(C^2)=0$, $C$ is nef.
Again, by Kodaira's lemma, $L=A+E$ for some ample $\mathbb{Q}$-divisor $A$ and effective $\mathbb{Q}$-divisor $E$.
Thus $(L \cdot C)=(A \cdot C)+(E \cdot C)>0$
since $C$ is nef.
This yields a contradiction.
If $(L^2)= 0$, $c_1(L)$ is proportional to $c_1(\mathcal{O}_X( C))$ by Hodge index theorem in Dolbeault cohomology (and hence Bott-Chern cohomology)  (see e.g. \cite[Chapter IV, Theorem 2.14]{BHPV}). 
%Thus the Ueda type of $(X,C,L)$ is the same as the Ueda type of $(X,C)$.
If $X$ is not Serre's example, by Neeman's result \cite[Article 1, Section 6; Article 2, Proposition 9.1]{Nee},
the Ueda type of $(X,C)$ is infinite.
Thus by Condition 2 in Definition \ref{ueda-def}, the Ueda classes of $(X,C,L)$ are independent of choice of system of local defining functions and local frames.
By Remark \ref{rem-first-chern}, the Ueda type of some positive multiple of $\mathcal{O}_X(C)$ is finite.
More precisely, the Hodge index theorem gives
$c_1(L)=a\,c_1(\mathcal O_X(C))$
for some \(a\in\mathbb R\). Since both classes are integral, nef, and \(c_1(\mathcal O_X(C))\neq0\), we have \(a\in\mathbb Q_{\geq 0}\). Writing \(a=p/q\) with \(p,q>0\), we obtain
$$ c_1(L^{\otimes q})=c_1(\mathcal O_X(C)^{\otimes p}). $$
By Remark~\ref{rem-first-chern}, the corresponding Ueda obstruction
classes agree under the natural identification. Since \((X,C,L)\) has
finite Ueda type, choosing the induced defining functions
and frames gives some \(k\) such that
$
u_k(X,C,L^{\otimes q})\neq 0.
$
On the other hand, since the Ueda type of \((X,C)\) is infinite, the
induced defining functions and frames give
$
u_k\bigl(X,C, \mathcal O_X(C)^{\otimes p}\bigr)=0.
$
By Condition~2 in Definition~\ref{ueda-def}, the Ueda class is
independent of the choice of defining functions and frames, yielding a
contradiction.

Consider now the case that $X$ is non-K\"ahler.
From now on, we will always assume that $X$ is  non-K\"ahler.
In this case, the Hodge index theorem (see e.g. \cite[Chapter IV, Theorem 2.14]{BHPV}) implies that $(L^2)=0$ and that
the first Chern class of $L$ is zero in
%is proportional to that of $C$ in
Dolbeault cohomology.
However, we emphasize that the nef class lies in Bott-Chern cohomology.
To emphasize this cohomology difference, we first record the cohomological consequence of non-Kählerness that will be used throughout the proof.
\paragraph{}
Over a non-K\"ahler surface,
%By \cite[Lemma 25.19]{OV24}, 
the natural map 
$H^{1,1}_{BC}(X, \mathbb{R}) \to H^{1,1}_{\dbar}(X, \mathbb{R})$
is not injective.

More precisely, we have
\begin{proposition}(\cite[Theorem 6.1]{Lam99}, \cite[Theorem 2.3]{OVV21})
\label{prop-lamari}
Let $X$ be a compact complex non-K\"ahler surface. Then
there exists a non-zero positive, $d$-exact $(1,1)$-current $\Theta$ on $X$.
Its class generates the kernel of
$$H^{1,1}_{BC}(X, \mathbb{R}) \to H^2_{\mathrm{dR}}(X, \mathbb{R}).$$
\end{proposition}
We will need the following identification of the two kernels.
\begin{proposition}
\label{prop-cohom-compare}
Let $X$ be a compact complex surface.
Let
\[
\rho:H^{1,1}_{BC}(X,\mathbb{R})\longrightarrow H^{1,1}_{\bar\partial}(X,\mathbb{R})
\]
and
\[
P:H^{1,1}_{BC}(X,\mathbb{R})\longrightarrow H^2_{\mathrm{dR}}(X,\mathbb{R})
\]
be the natural maps. We have that
\[
\ker\rho=\ker P.
\]
\end{proposition}
\begin{proof}
Since \(X\) is a compact complex surface, the Frölicher spectral sequence
degenerates at \(E_1\). In degree two, this gives
\[
\operatorname{Gr}^{2,0}_{F}H^2_{\mathrm{dR}}(X,\mathbb{C})
\simeq H^{2,0}_{\bar\partial}(X),\qquad
\operatorname{Gr}^{1,1}_{F}H^2_{\mathrm{dR}}(X,\mathbb{C})
\simeq H^{1,1}_{\bar\partial}(X),\qquad
\operatorname{Gr}^{0,2}_{F}H^2_{\mathrm{dR}}(X,\mathbb{C})
\simeq H^{0,2}_{\bar\partial}(X),
\]
and
\[
F^2H^2_{\mathrm{dR}}(X,\mathbb{C})
\simeq H^{2,0}_{\bar\partial}(X).
\]

Suppose first that \([\alpha]_{BC}\in\ker P\). Write
\(\alpha=d\gamma\), where \(\gamma\) is a real \(1\)-form, and decompose
\(\gamma=\gamma^{1,0}+\gamma^{0,1}\). Since \(\alpha\) has type \((1,1)\),
its \((0,2)\)-component is
\(\bar\partial\gamma^{0,1}=0\). Hence
\([\gamma^{0,1}]_{\bar\partial}\in H^{0,1}_{\bar\partial}(X)\).
The \((1,1)\)-component of \(d\gamma\) is
\(\partial\gamma^{0,1}+\bar\partial\gamma^{1,0}\). Since the Frölicher
spectral sequence degenerates at \(E_1\), \(d_1^{0,1}=0\), and therefore
\([\partial\gamma^{0,1}]_{\bar\partial}=0\), i.e.
\(\partial\gamma^{0,1}\) is \(\bar\partial\)-exact. It follows that
\([\alpha]_{\bar\partial}=0\), so \(\ker P\subset\ker\rho\).

Conversely, suppose that \([\alpha]_{BC}\in\ker\rho\), and let
\(a=P([\alpha]_{BC})\). Since \(\alpha\) has type \((1,1)\), the class
\(a\) belongs to \(F^1H^2_{\mathrm{dR}}(X,\mathbb{C})\). Under the
identification induced by the \(E_1\)-degeneration,
\[
F^1H^2_{\mathrm{dR}}(X,\mathbb{C})/
F^2H^2_{\mathrm{dR}}(X,\mathbb{C})
\simeq H^{1,1}_{\bar\partial}(X),
\]
the image of \(a\) is precisely \([\alpha]_{\bar\partial}\). Since
\([\alpha]_{\bar\partial}=0\), we obtain
\(a\in F^2H^2_{\mathrm{dR}}(X,\mathbb{C})\).

On the other hand, \(a\) is a real de Rham cohomology class. Hence its complex
conjugate is equal to \(a\), while complex conjugation sends
\(F^2H^2_{\mathrm{dR}}(X,\mathbb{C})\) to
\(\overline{F^2H^2_{\mathrm{dR}}(X,\mathbb{C})}\). Thus
\[
a\in F^2H^2_{\mathrm{dR}}(X,\mathbb{C})
\cap
\overline{F^2H^2_{\mathrm{dR}}(X,\mathbb{C})}.
\]
We claim that this intersection is zero. Indeed, let \(a\) be represented
by a holomorphic \(2\)-form \(\eta\). Since $a$ is real, \(a\) is 
is also represented by a \((0,2)\)-form $\overline{\eta}$. In particular,
\([\eta\wedge\overline{\eta}]=[\eta \wedge \eta]=0\) in \(H^4_{\mathrm{dR}}(X,\mathbb{C})\).
Hence
$
\int_X\eta\wedge\overline{\eta}=0.
$
Since \(\eta\wedge\overline{\eta}\) is a non-negative volume
form, it follows that \(\eta=0\), and therefore \(a=0\). Hence
\([\alpha]_{BC}\in\ker P\), proving
$
\ker\rho=\ker P.
$
\end{proof}
Therefore, 
for a compact non-Kähler surface, the kernel of \(P\) is one-dimensional. Consequently, if \([\Theta]_{BC}\) is a nonzero generator of this kernel, then
\[
\ker\left(H^{1,1}_{BC}(X,\mathbb{R})\longrightarrow H^{1,1}_{\bar\partial}(X,\mathbb{R})\right)
=
\mathbb{R}\{\Theta\}_{BC}.
\]
We now return to the classification of $X$. By the preceding discussion and Proposition \ref{prop-cohom-compare}, 
%after replacing $L$
%by a suitable positive tensor power, 
there exist %$a\in\mathbb N$
%and 
$b\in\mathbb R$ such that
$$c_1(L)=%a c_1(\mathcal{O}(C))+
b \{ \Theta \} \in H^{1,1}_{BC}(X,\mathbb{R}).$$
Note that $L%\otimes\mathcal O_X(-aC)
\in\operatorname{Pic}^0(X)$.
If $b=0$, $L$ is Hermitian flat and the Ueda type of $(X,C,L)$ is not finite.
We obtain a contradiction.
Thus $b \neq 0$.
Up to changing $\Theta$ by a positive multiple, we may assume that $b=1$.
% and $X$ is not a Hopf surface, by Neeman's classification \cite[Article 2, Section 9]{Nee},
%the Ueda type of $(X,C)$ is infinite.
%Thus by Condition 2 in Definition \ref{ueda-def}, the Ueda classes of $L$ is independent of choice of system of local defining functions and local frames.
%By Remark \ref{rem-first-chern}, the Ueda type of some positive multiple of $\mathcal{O}_X(C)$ is finite.
%Contradiction with the condition that the Ueda type of $(X,C)$ is infinite. 
%the Ueda type of the pair $(X,C,L)$ is the same as the Ueda type of $(X,C)$ which is classified by \cite[Article 2, Section 9]{Nee}.
%In particular, $X$ is a non-diagonal Hopf surface.

%In what follows, we assume that $b \neq 0$ and that the Ueda type of $(X,C)$ is infinite.
%Since $\Theta$ is not $\partial \dbar$-exact, $a > 0$. 

%Consider the Neron-Serveri part in the corresponding spaces
%$H^{1,1}_{BC} \to H^{1,1}_{\dbar} \to H^{1,1}_{A}(X)$.
%By Hodge index theorem and the duality between Bott-Chern and Aepplu cohomology,
%these are isomorphic.

We split the discussion according to the algebraic dimension of $X$.
\paragraph{}
%Consider the 
\textbf{Case 1:} $a(X)=1$.

In this case, there exists an elliptic fibration $f: X \to B$.
We claim that $C$ is not $f$-horizontal; indeed,
otherwise, 
$$((C+ F)^2)=2(C \cdot F)+(C^2)+(F^2)>0$$
where $F$ is a fiber of $f$ since $(C^2)=0$.
This contradicts the non-Kählerness of $X$ by
\cite[Theorem 8]{Koi66}.

%It is well-known (see \cite[Lemma 1]{Bri94}) that an elliptic fibration non K\"ahler surface has biholomorphic (reduced) fibers where singular fibers are multiple elliptic curves.
We need three facts about non-Kähler elliptic fibrations. The first two allow us to remove the singular fibres after a finite cyclic base change.
They allow us, after passing to a finite étale cover, to replace $X$
by a smooth elliptic bundle and to understand the original multiple
fibres.
\begin{lemma}
\label{lem-base-change}
Let \(f\colon X\to B\) be a non-K\"ahler minimal elliptic surface. Then there exists a
finite cyclic morphism \(q\colon B'\to B\) such that, if
\[
X_{B'}:=\operatorname{Norm}\bigl(X\times_B B'\bigr)
\]
denotes the normalization of the base change, the induced elliptic fibration
\[
f_{B'}\colon X_{B'}\to B'
\]
is smooth. Moreover, \(f_{B'}\) is locally trivial in the analytic topology.
\end{lemma}

\begin{proof}
By \cite[Lemma 2]{Bri94}, there exists a finite cyclic morphism
\(q\colon B'\to B\) such that the normalization
\[
X_{B'}=\operatorname{Norm}\bigl(X\times_B B'\bigr)
\]
of the base change is smooth over \(B'\) and $X_{B'}$ is a non-K\"ahler complex surface. By \cite[Lemma 1]{Bri94},
all fibres of \(f_{B'}\) are isomorphic. Hence \(f_{B'}\) is locally trivial
in the analytic topology by the Fischer--Grauert theorem.
\end{proof}
By \cite[Appendix]{BU}, we have the following local structure theorem near the multiple fibers.
\begin{lemma}
\label{lem-local-model}
Let \(f\colon X\to B\) be a non-K\"ahler elliptic surface, and suppose that
\(f^{-1}(b)=mC\) is a multiple fibre, where \(C\) is a smooth elliptic curve.
After shrinking a neighbourhood of \(b\), the neighbourhood of \(mC\) is
obtained by a logarithmic transformation of order \(m\) as follows.

Let \(T\) be a smooth elliptic curve, written as \(T=\mathbb C/\Lambda\), and choose
a point \(a\in T[m]\) of exact order \(m\). Let
\(\widetilde a\in\mathbb C\) be a lift of \(a\), so that
\(m\widetilde a\in\Lambda\), and put
\[
\zeta_m=\exp\left(\frac{2\pi\sqrt{-1}}{m}\right).
\]
Let \(\Delta_s\) be a sufficiently small disk with coordinate \(s\), and let
\(G\simeq\mathbb Z/m\mathbb Z\) act on \(T\times\Delta_s\) by
\[
g(x,s)=\bigl(x+a,\zeta_m s\bigr).
\]
The action is free, and the quotient
$
X_{\mathrm{loc}}
:=(T\times\Delta_s)/G
$
is an elliptic surface over \(\Delta_t\), where
$
t=s^m.
$
Its central fibre is
$
f^{-1}(0)=mC$, $
C\simeq T/\langle [a] \rangle.
$
Hence the quotient map in Lemma \ref{lem-base-change}
$
\pi: X_{B'}\longrightarrow X
$
is finite \'{e}tale.
\end{lemma}
%Denote the $\pi: X_{B'} \to X$ the cyclic cover.

The third describes the Picard group by the following result of Teleman \cite{Tel98}.
%Recall the following result of Teleman \cite{Tel98}.
\begin{proposition}
\label{prop-pic-0}
Let \(f_{B'}:X_{B'}\to B'\) be a smooth elliptic fibration of a non-K\"ahler surface $X_{B'}$. Then there is a positive integer \(d\), a line bundle \(\Theta\in\operatorname{Pic}^d(B')\) of degree \(d\), and a complex number \(\tau\) with \(|\tau|>1\) such that
\[
X_{B'}\cong \Theta^*/\langle\tau\rangle,
\]
where \(\Theta^*\) is the principal \(\mathbb{C}^*\)-bundle corresponding to \(\Theta\), and \(\langle\tau\rangle\) is the subgroup of \(\mathbb{C}^*\) generated by \(\tau\), acting via the standard action. 
Moreover, every element of \(P\in \operatorname{Pic}^0(X_{B'})\) is uniquely of the form
\[
P=f_{B'}^*M\otimes L_\alpha,
\]
where \(M\in\operatorname{Pic}^0(B')\) and \(\alpha\in\mathbb{C}^*\). Here \(L_\alpha\) is the line bundle obtained from the trivial line bundle \(\Theta^*\times\mathbb{C}\) by the action
\[
(z,v)\longmapsto(\tau z,\alpha v).
\]
\end{proposition}
To finish the classification of the case $a(X)=1$, we need to show that Ueda type of $(X, C, L)$ is not finite.
%It is equivalent to show that Ueda type of $(X, C, L)$ is infinite by the assumption that Ueda type of $(X,C)$ is infinite.
Note that $\pi^*(L) \in \mathrm{Pic}^0(X_{B'})$.

We split the discussion according to whether $C$ is a smooth fiber of $f$.
\paragraph{}
\textbf{Subcase 1:} 
%First consider the case that 
$C$ is a smooth fiber of $f$.

In this case, by construction of cyclic cover, we may assume that $B' \to B$ is \'etale near $f(C)$.
Since the Ueda type is determined by an arbitrarily small neighborhood of $C$ in $X$.
By abuse of notation, we identify $C$ as a smooth fiber of $f_{B'}$.
Since the Ueda type is determined by an arbitrarily small
neighborhood of $C$, we may work on a neighborhood over which
$f_{B'}$ is trivial. There exists a coordinate neighborhood $U$ of $f_{B'}(C)$ such that $f_{B'}$ is identified to the projection $U \times F \to U$. 
To show that the Ueda type of $(X_{},C)$ is infinite,
it is enough to prove that the Ueda type of $(X_{B'},C)$ is infinite.
Without loss of generality, applying Proposition \ref{prop-pic-0}, we may assume that $M$ and $\Theta$ are trivial on $U$.
By definition, $L_\alpha$ is the quotient of $U \times \mathbb{C}^* \times \mathbb{C}$
by identifying
$(z,u,v)$ with $(z, \tau u, \alpha v)$.
Since $L_\alpha$ is the pull back of a flat line bundle from the projection $U \times F \to F$,
the Ueda type of $(X, C, L)$ is not finite.
\paragraph{}
\textbf{Subcase 2:} 
%First consider the case that 
%$C$ is a smooth fiber of $f$.
%Now consider the case that 
$C$ is a multiple fiber of $f$.

Recall the following classical lemma which is well-known to experts.
\begin{lemma}
\label{lem-inj-topo}
Let \(q\colon X'\to X\) be a finite étale morphism of complex manifolds, and let
\(C\subset X\) be a smooth compact submanifold. Set \(C'=q^{-1}(C)\), and denote
by \(q_C\colon C'\to C\) the induced finite étale morphism. Then
\[
N_{C'/X'}\simeq q_C^*N_{C/X},
\]
and the pull-back morphism
\[
q_C^*\colon H^1(C,S^d N^*_{C/X})
\longrightarrow
H^1(C',S^d N^*_{C'/X'})
\]
is injective for any $d \in \mathbb{N}$.
\end{lemma}

\begin{proof}
Since \(q\colon X'\to X\) is étale, the differential induces an isomorphism
$
q^*T_X\simeq T_{X'}.
$
Restricting to \(C'\) and using the commutative diagram
\[
\begin{array}{ccccccccc}
0 &\longrightarrow& T_{C'} &\longrightarrow& T_{X'}|_{C'}
&\longrightarrow& N_{C'/X'} &\longrightarrow& 0\\
&& \big\downarrow && \big\downarrow && \big\downarrow &&\\
0 &\longrightarrow& q_C^*T_C &\longrightarrow& q_C^*T_X|_C
&\longrightarrow& q_C^*N_{C/X} &\longrightarrow& 0,
\end{array}
\]
we obtain
$
N_{C'/X'}\simeq q_C^*N_{C/X}.
$

It therefore suffices to prove that
\[
q_C^*\colon H^1(C,S^d N^*_{C/X})
\longrightarrow
H^1(C',q_C^*S^d N^*_{C/X})
\]
is injective.

Let \(r=\deg q_C\). Since \(q_C\) is finite, there is a trace morphism
\[
\operatorname{Tr}_{q_C}\colon
(q_C)_*q_C^* S^d N^*_{C/X}\longrightarrow S^d N^*_{C/X}.
\]
It satisfies
\[
\operatorname{Tr}_{q_C}\circ q_C^*
=r \,\operatorname{id}_{S^d N^*_{C/X}}.
\]
Consequently, on cohomology,
\[
\operatorname{Tr}_{q_C}\circ q_C^*
=r \,\operatorname{id}_{H^1(C,S^d N^*_{C/X})}.
\]
Since \(H^1(C,S^d N^*_{C/X})\) is a complex vector space, \(q_C^*\) is injective.
\end{proof}
Apply Lemma \ref{lem-inj-topo} to $\pi: X_{B'} \to X$.
By the functoriality of Ueda classes, it is enough to prove that the Ueda type of $(X_{B'},C', \pi^*(L))$ is not finite. 
This is the case by the previous discussion of Subcase 1.

%Up to a possible shrinking of $U$, Proposition \ref{prop-pic-0} implies that
%$\pi^*(L-aC) \simeq L_\alpha$
%for some $\alpha \in \mathbb{C}^*$ on $f_{B'}^{-1}(U)$.
%Let $F$ be a smooth fiber of $f_{B'}$.
%Consider the holomorphic map
% is constant.

%On the other hand, identify $F$ as the center fiber $f^{-1}(0) \simeq f_{B'}^{-1}(0)$ and let 
%Identify  $f_{B'}^{-1}(U) \simeq U \times F$ with projection $p: U \times F \to F$.

%Since $\Theta$ is closed positive $(1,1)$-current, 
%$\pi^* \Theta$ is well defined such that
%$$c_1(\pi^* L)=a c_1(\pi^* \mathcal{O}(C))+b \{\pi^* \Theta \} \in H^{1,1}_{BC}(X_C,\mathbb{R}).$$
%In particular,
%$$c_1(\pi^* L)-a c_1(\pi^* \mathcal{O}(C)) \in \mathrm{Ker}(H^{1,1}_{BC}(X_C,\mathbb{R})\to H^{2}_{dR}(X_C,\mathbb{R})).$$
%The construction of $C \to B$ is obtained by cyclic cover of $B$ with suitable number of points.
%In particular, the points can be moved to avoid a small open neighborhood of any regular value of $f$.

%The line bundle is minimal resolution of singular point of blow down zero section in projective bundle which is an elliptic curve of negative self-intesection.
%In particular, the section at infinity is elliptic curve.
%The line bunlde is trivial near section at infinity.
%The Ueda class thus vanishes.

\paragraph{}
%Consider the 
\textbf{Case 2:} $a(X)=0$.

%Now consider the case that $a(X)=0$.
By Kodaira--Enriques classification, $X$ is of class VII.
If $b_2(X)=0$, $X$ is either a Hopf surface or Inoue surface by \cite{Bo83, LYZ90}.
Over an Inoue surface, there is no curve.
Thus $X$ cannot be an Inoue surface.
If $b_2(X)>0$, by \cite[Proposition 4.12]{Eno81} or \cite[Theorem 10.2]{Nak84}, $X$ is a parabolic Inoue surface. 
However a parabolic Inoue surface contains a unique elliptic curve $C$ and the self-intersection of $C$ satisfies $(C^2)<0$.
Thus $X$ cannot be a parabolic Inoue surface.
%By \cite[Theorem 10]{Toma08}, $\{\Theta\}$ is some multiple of $c_1(\mathcal{O}(Z))$ in Bott-Chern cohomology.
%Since the curves $C$ and $Z$ are disjoint,
%the Ueda type of $(X,C,L)$ is the same as Ueda type of $(X,C)$ which is infinite.

Consider the case of Hopf surfaces in more detail.
%Note that $c_1(L)$ is proportional to $c_1(\mathcal{O}(C))$ (in Bott-Chern cohomology). 
%Thus, it is enough to consider the Ueda type of the pair $(X,C)$.

Assume first that $X$ is a primary Hopf surface.
If $X$ is diagonal, $a(X)=1$ and by previous discussion, the Ueda type of $(X,C,L)$ is not finite. 
%then $\mathcal{O}(C)$ is semi-positive. By
%\cite[Theorem 1.2]{Koi22}, the Ueda type of $(X,C)$ is infinite.
If $X$ is non-diagonal, then $X$ contains a unique elliptic curve, namely $C$. In this case, the Ueda type of $(X,C)$ is finite by Proposition \ref{prop:ueda-computation-summary}.

Now assume that $X$ is a secondary Hopf surface. Recall that every secondary Hopf surface admits a finite étale cover
$
\pi\colon X'\longrightarrow X
$
such that $X'$ is a primary Hopf surface. Since $\pi|_{\pi^{-1}(C)}$ is a finite étale cover of the elliptic curve $C$, every connected component of $\pi^{-1}(C)$ is again an elliptic curve.
Suppose first that $X'$ is non-diagonal. Since $X'$ contains a unique elliptic curve, $\pi^{-1}(C)$ is connected. Hence Lemma~\ref{lem-inj-topo} implies that the Ueda type of $(X,C,L)$ is finite.
Finally, suppose that $X'$ is diagonal. In this case, since the Ueda type of $(X,C,L)$ depends only on an arbitrarily small open neighborhood of $C$, we may apply Lemma~\ref{lem-inj-topo} to an open neighborhood of $C$ in $X$ and an open neighborhood of a connected component of $\pi^{-1}(C)$ in $X'$. Since $X'$ is diagonal, the Ueda type of this connected component with respect to the pull-back line bundle is not finite. Therefore, by Lemma~\ref{lem-inj-topo}, the Ueda type of $(X,C,L)$ is also not finite.
\paragraph{}
This completes the proof of classification.
\paragraph{}
Note that the above argument does not use in an essential way the assumption that \(C\) is an elliptic curve. Similar arguments show the following.

\begin{proposition}
\label{prop-higher-genus}
Let \(X\) be a minimal non-projective surface and let \(C\) be a curve of genus at least \(2\) in \(X\). Then \((C^2)\neq 0\). 
\end{proposition}
\begin{proof}
We argue by contradiction and assume that \((C^2)=0\).

First consider the case \(a(X)=1\), with elliptic fibration \(f\colon X\to B\). Since \(C\) is not an elliptic curve, \(C\) is \(f\)-horizontal. Let \(F\) be a smooth fiber of \(f\). Then \( (C\cdot F)>0\), and hence
\[
((C+F)^2)=(C^2)+2(C\cdot F)+(F^2)=2(C\cdot F)>0.
\]
This contradicts the fact that a compact non-projective surface has no curve of positive self-intersection. Hence \(a(X)\neq 1\).

It remains to consider the case \(a(X)=0\). 
Consider first the case that $X$ is non-K\"ahler.
Thus \(X\) is a surface of class VII. By \cite[Chapter IV, Theorem 2.7]{BHPV}, since \(h^{2,0}(X)=0\), the intersection form on \(H^2_{\mathrm{dR}}(X,\mathbb{R})\) is negative definite. Since
\[
c_1(\mathcal O_X(C))^2=(C^2)=0,
\]
we obtain \(c_1(\mathcal O_X(C))=0\) in \(H^2_{\mathrm{dR}}(X,\mathbb{R})\). Consequently, \((K_X\cdot C)=0\). Together with \((C^2)=0\), the adjunction formula gives
\[
2g(C)-2=(C^2)+(K_X\cdot C)=0,
\]
which contradicts \(g(C)\geq 2\). (In fact, a complete classification of effective curves $C$ with $(C^2)=0$ on surfaces of class $VII_0$ is given in \cite[Lemma 2.10]{Nak84}.)

Consider now the case that $X$ is K\"ahler. Then by Kodaira--Enriques classification, $X$ is a torus or K3 surface. In particular, \((K_X\cdot C)=0\).
A similar intersection-number calculation gives a contradiction.
\end{proof}
\begin{remark}
Under the assumption of Proposition \ref{prop-higher-genus}, suppose that
$L$ is a nef line bundle on $X$. Then the Ueda type of $(X,C,L)$
cannot be finite.

Indeed, otherwise, in order for the Ueda class to be well-defined, we need
to assume that $L|_C$ is unitary flat. In particular, $(L \cdot C)=0$.

Consider first the case \(a(X)=1\). 
Let $F$ be a smooth fiber of the associated elliptic fibration. Since
$C$ is of genus at least $2$, $C$ is not contained in a fiber, and hence
$(C \cdot F)>0$. Thus
$((C+mF)^2)=2m(C \cdot F)+(C^2)>0$
for $m$ sufficiently large. This contradicts the non-projectivity of
$X$.

Now consider the case that $X$ is K\"ahler but not projective with $a(X)=0$. The same
intersection number calculation as in the proof of
Proposition \ref{prop-higher-genus} gives a contradiction.

Finally consider the case that $X$ is non-K\"ahler with $a(X)=0$.
Then \(X\) is of class
\(VII_0\), and hence contains no curve of genus at least \(2\) by
\cite[Lemma~(2.2.2)]{Nak84}.
\end{remark}
\begin{remark}
As shown in \cite[Article 2, Remark 6.3, p.~115]{Nee}, for some fixed genus of $C$, there exist infinitely many deformation
classes of surfaces of general type $X$ containing a curve $C$ such
that the
Ueda type of \((X,C)\) is finite.

In particular, a complete classification of projective surfaces
containing a curve of genus at least \(2\) along which the Ueda type is
finite appears to be out of reach.
%As shown in \cite[Article 2, Section 8, Page 119]{Nee}, it seems to be hard to obtain a complete list of projective surface containing a curve of genus at least 2 along which the Ueda type is finite.
\end{remark}
%Another proof of nef not semipositive line bundle on Hopf surface is given by \cite{CT22}.
%In Dolbeault cohomology not Bott-Chern cohomology.
%If $X$ admits an elliptic fibration $f: X \to B$.
%If $f(C)$ is $B$, $(C+F)^2 >0$ where $F$ is a smooth fiber.
%By \cite[Theorem 8]{Koi66} implies that $X$ is projective.
%Contradiction.

%Thus $C$ is contained in a fiber.
%If $X$ is minimal, $C$ is a multiple fiber.
%Up to a mulitple, $C$ can move in familly.
For the convenience of the reader, we recall the following higher-dimensional construction of a non-minimal projective manifold $X$ containing a smooth divisor $D$ such that the Ueda type of $(X,D,\mathcal O_X(D))$ is finite. In particular, by \cite[Theorem 1.2]{Koi22}, $\mathcal O_X(D)$ is not semi-positive.

\begin{example}
\label{example-decomp}
Let $X'$ be a smooth projective variety of dimension $n\geq 2$, and let $D'\subset X'$ be a smooth very ample divisor. Assume that $H^1(X',\mathcal O_{X'})\neq 0$ and $H^1(X',\mathcal O_{X'}(D'))=0$. The short exact sequence
\[
0\to\mathcal O_{X'}\to\mathcal O_{X'}(D')\to\mathcal O_{D'}(D')\to0
\]
induces a surjective map $H^0(D',\mathcal O_{D'}(D'))\to H^1(X',\mathcal O_{X'})$. Hence a general element of $H^0(D',\mathcal O_{D'}(D'))$ has non-zero image in $H^1(X',\mathcal O_{X'})$. On the other hand, since $D'$ is very ample, $\mathcal O_{D'}(D')$ is very ample, and hence Bertini's theorem shows that a general member of $|\mathcal O_{D'}(D')|$ is smooth. Therefore, there exists $s\in H^0(D',\mathcal O_{D'}(D'))$ with smooth zero divisor $Z=(s)_0$ such that the image of $s$ in $H^1(X',\mathcal O_{X'})$ is non-zero.

Let $\pi\colon X=\operatorname{Bl}_Z X'\to X'$ be the blow-up, and let $D$ be the strict transform of $D'$. Then $N_{D/X}\simeq\mathcal O_D$, where the trivialization is induced by $s$. We claim that if $\delta_X(1)\in H^1(X,\mathcal O_X)$ and $\delta_{X'}(s)\in H^1(X',\mathcal O_{X'})$ denote the connecting classes associated with
\[
0\to\mathcal O_X\to\mathcal O_X(D)\to\mathcal O_D(D)\to0
\quad\text{and}\quad
0\to\mathcal O_{X'}\to\mathcal O_{X'}(D')\to\mathcal O_{D'}(D')\to0,
\]
respectively, then $\delta_X(1)=\pi^*\delta_{X'}(s)$. In particular, $\delta_X(1)\neq0$, since $\pi^*\colon H^1(X',\mathcal O_{X'})\to H^1(X,\mathcal O_X)$ is an isomorphism.

We postpone the proof of the claim to the end of the example.

Note that the same calculation as in \cite[Article 2, Section 5, Page 115]{Nee} identifies the cup product with the first Ueda class of $(X,D,\mathcal O_X(D))$ with the composition
\[
H^1(D,\mathcal O_D(D))\to H^1(X,\mathcal O_X)\to H^1(D,\mathcal O_D),
\]
where the second map is induced by restriction.

To show that the first Ueda class is non-zero, it is therefore enough, by the above claim, to show that the restriction map $H^1(X,\mathcal O_X)\to H^1(D,\mathcal O_D)$ is injective. This follows from the injectivity of the map $H^1(X',\mathcal O_{X'})\to H^1(D',\mathcal O_{D'})$. Indeed, the sequence
\[
H^1(X',\mathcal O_{X'}(-D'))\to H^1(X',\mathcal O_{X'})\to H^1(D',\mathcal O_{D'})
\]
is exact, and
\[
H^1(X',\mathcal O_{X'}(-D'))\simeq H^{n-1}(X',\mathcal O_{X'}(K_{X'}+D'))=0
\]
by Serre duality and Kodaira vanishing. 

Since $\mathcal O_D(D)\simeq\mathcal O_D$, the intersection number $(D\cdot C)$ is zero for every curve $C\subset D$, while $(D\cdot C)\geq 0$ for every curve $C\not\subset D$ because $D$ is effective. Hence $\mathcal O_X(D)$ is nef.

\end{example}

\noindent\textit{Proof of the claim.}

Choose an open cover $\{U_i\}$ of $X'$ such that $D'\cap U_i=(f_i)$. Let $e_i$ denote the local generator of $\mathcal O_{X'}(D')$ corresponding to $1/f_i$.
On $U_{ij}$, write $f_j=g_{ij}f_i$, so that $e_j=g_{ij}^{-1}e_i$. 
The section $s$ on $D'\cap U_i$ can be written as $s=h_i e_i|_{D'\cap U_i}$ for some $h_i\in\mathcal O_{D'}(D'\cap U_i)$. The zero divisor $Z=(s)_0$ is then defined on $D'\cap U_i$ by $h_i=0$. By choosing arbitrary extensions, we may regard $h_i$ as elements of $\mathcal O_{X'}(U_i)$. Then $h_i e_i$ is a local lift of $s$ to $\mathcal O_{X'}(D')$, and hence
\[
\delta_{X'}(s)=\left[\{h_i e_i-h_j e_j\}\right]\in H^1(X',\mathcal O_{X'}).
\]
%$g_{ij}h_i=h_j+f_ja_{ij}$
%and $a_{ij}\in\mathcal O_{D'}(U_{ij})$
%Note that by the above identification, $h_i/f_i$ defines a local section of $\mathcal O_{U_i}(D')$.
%On $U_{ij}$, we may write $f_j=g_{ij}f_i$ some $g_{ij}\in\mathcal O_{X'}(U_{ij})$.

On the blow-up $\pi\colon X\to X'$, let
\[
V_i^{(t)}=\operatorname{Spec}\mathcal O_{U_i}[t_i]/(f_i-h_it_i),
\qquad
V_i^{(u)}=\operatorname{Spec}\mathcal O_{U_i}[u_i]/(h_i-f_iu_i),
\]
where $f_i,h_i\in\mathcal O_{X'}(U_i)$ are the downstairs functions, while $t_i$ and $u_i$ are the coordinates on the blow-up. On $V_i^{(t)}$, we have $f_i=h_it_i$, and hence
\[
D\cap V_i^{(t)}=(t_i),\qquad E\cap V_i^{(t)}=(h_i).
\]
On $V_i^{(t)}$, let $\eta_i$ be the local frame of $\mathcal O_X(E)$ corresponding to the canonical section defined by $h_i$, and let $\varepsilon_i$ be the local frame of $\mathcal O_X(D)$ determined by
$
\pi^*e_i=\varepsilon_i\otimes\eta_i.
$
On $U_{ij}$, write
\[
h_i e_i-h_j e_j=a_{ij}f_i e_i
\]
for $a_{ij}\in\mathcal O_{X'}(U_{ij})$. Since $e_j=g_{ij}^{-1}e_i$ and $f_i=h_it_i$, we obtain
\[
\frac{h_j}{\pi^* g_{ij}h_i}=1-\pi^* a_{ij}t_i.
\]
As $h_i\eta_i=h_j\eta_j$, it follows that
\[
\varepsilon_j=(1-\pi^* a_{ij}t_i)\varepsilon_i,
\qquad
\varepsilon_i-\varepsilon_j=\pi^* a_{ij}t_i\varepsilon_i.
\]
Since $t_i\varepsilon_i$ is the canonical section of $\mathcal O_X(D)$ defining $D$, the inclusion $\mathcal O_X\hookrightarrow \mathcal O_X(D)$ identifies $1$ with $t_i\varepsilon_i$. Hence $\varepsilon_i-\varepsilon_j=\pi^* a_{ij}t_i\varepsilon_i$ corresponds to $\pi^*a_{ij}\in\mathcal O_X(V_i^{(t)}\cap V_j^{(t)})$. Therefore the \v{C}ech cocycle representing $\delta_X(1)$ is $\{\pi^*a_{ij}\}$, and consequently $\delta_X(1)=\pi^*\delta_{X'}(s)$.
%Since $t_i\varepsilon_i$ is the canonical section of $\mathcal O_X(D)$ defining $D$, the corresponding $\mathcal O_X$-valued \v{C}ech cocycle is $\{\pi^*a_{ij}\}$. 
%Hence
%$
%\delta_X(1)=\pi^*\delta_{X'}(s).
%$

(Note that, by the definition of the first Ueda class, the restriction of the \v{C}ech cocycle $\{\pi^*a_{ij}\}$ to $V_i^{(t)}\cap V_j^{(t)}\cap D$ represents the first Ueda class of $(X,D,\mathcal O_X(D))$.)
\qed

Department of Mathematics, Graduate School of Science, Osaka Metropolitan University, 3-3-138 Sugimoto, Osaka 558-8585, Japan.\\
Email address:  y25161q@omu.ac.jp.  
  \end{document}